\documentclass{article}[11pt]
\usepackage{graphicx}
\usepackage[utf8]{inputenc}
\usepackage[T1]{fontenc}
\usepackage{amsmath}
\usepackage{amsthm}
\usepackage{amsfonts}
\usepackage{amssymb}
\usepackage{enumitem}
\usepackage[all,cmtip]{xy}
\usepackage{spectralsequences}
\usepackage[normalem]{ulem}
\usepackage{tikz-cd}

\usepackage{natbib}

\usepackage{url}
\usepackage{slashed}
\usepackage{bbm}
\usepackage{tikz-cd}
\usepackage{mathtools}
\usepackage{esint}

\usepackage{ifthen}
\usepackage{xspace}
\usepackage{fancyhdr}
\usepackage{aliascnt}
\usepackage[textsize=small]{todonotes}
\usepackage{bookmark}
\usepackage{caption}

\usetikzlibrary{calc}
\usepackage{cancel}

\usepackage[margin=1in]{geometry}

\usepackage[T1]{fontenc}

\usepackage{mlmodern}

\usepackage{microtype}

\usepackage{setspace}
\tikzset{curve/.style={settings={#1},to path={(\tikztostart)
    .. controls ($(\tikztostart)!\pv{pos}!(\tikztotarget)!\pv{height}!270:(\tikztotarget)$)
    and ($(\tikztostart)!1-\pv{pos}!(\tikztotarget)!\pv{height}!270:(\tikztotarget)$)
    .. (\tikztotarget)\tikztonodes}},
    settings/.code={\tikzset{quiver/.cd,#1}
        \def\pv##1{\pgfkeysvalueof{/tikz/quiver/##1}}},
    quiver/.cd,pos/.initial=0.35,height/.initial=0}

\newtheorem{theorem}{Theorem}[section]
\newtheorem{prop}[theorem]{Proposition}
\newtheorem{lemma}[theorem]{Lemma}
\newtheorem{cor}[theorem]{Corollary}

\theoremstyle{definition}
\newtheorem{definition}[theorem]{Definition}

\newtheorem{rem}[theorem]{Remark}

\newtheorem{manualtheoreminner}{Theorem}    
\newenvironment{manualtheorem}[1]{%
  \renewcommand\themanualtheoreminner{#1}%
  \manualtheoreminner
}{\endmanualtheoreminner}

\newcommand{\coker}{\textnormal{coker\,}}

\newcommand{\Pin}{\textnormal{Pin}}

\newcommand{\m}[1]{{\mathrm{#1}}}

\newcommand{\cat}[1]{\textbf{#1}}

\newcommand{\und}[1]{\hspace{#1}\text{and}\hspace{#1}}

\newcommand{\Comment}[2][\empty]{\ifthenelse{\equal{#1}{\empty}}{\todo[color=gray!10]{#2}\ }{\todo[color=gray!10,#1]{#2}}}

\usepackage{hyperref}

\title{Non-Existence of Smooth Full-Holonomy Cayley Fibrations}

\author{
Jianfeng Lin\thanks{Tsinghua University. Email: \texttt{linjian5477[]mail.tsinghua.edu.cn}}
\and
Viktor Majewski\thanks{University of Waterloo. Email: \texttt{viktor.majewski[]uwaterloo.ca}}
\and
Jacek Rzemieniecki\thanks{Humboldt Universit\"at zu Berlin. Email: \texttt{jacek.rzemieniecki[]hu-berlin.de}}
}

\date{\today}

\begin{document}

\maketitle

\begin{abstract}
    We prove that every Cayley fibration of a compact torsion-free $\mathrm{Spin}(7)$-manifold with full holonomy must have singular fibers. This confirms a long-standing expectation in the study of calibrated fibrations with exceptional holonomy. Our argument first uses the rigidity of the $\mathrm{Spin}(7)$-structure to reduce any hypothetical nonsingular Cayley fibration to two possible topological configurations. We then exclude both by proving a new spinnability theorem for smooth fiber bundles over simply-connected $4$-manifolds with fiber homeomorphic to the elliptic surfaces $E(2)$ or $E(4)$. The $E(4)$ case, which constitutes the main new difficulty, is established by combining families Seiberg--Witten theory with parametrized equivariant homotopy theory and equivariant $K$-theory.
\end{abstract}

\section{Introduction}
\label{Introduction}

Manifolds with exceptional holonomy provide a fundamental class of Riemannian manifolds with rigid geometric structure. In dimension eight, torsion-free $\m{Spin}(7)$-manifolds are characterized by a parallel Cayley calibration, whose calibrated submanifolds are volume-minimizing $4$-dimensional submanifolds. This naturally leads to the question of whether such manifolds contain global fibrations by Cayley submanifolds, and to what extent the existence of such a fibration is compatible with full $\m{Spin}(7)$-holonomy.

These fibrations are the $\m{Spin}(7)$-analogues of special Lagrangian fibrations in Calabi--Yau geometry and coassociative fibrations in $G_2$-geometry. In the recent years there has been a programme initiated by Donaldson \cite{donaldson2016adiabaticlimitscoassociativekovalevlefschetz} of systematically studying coassociative and Cayley fibrations of $G_2$- and $\rm{Spin(7)}$-manifolds, respectively. A guiding principle in all these settings is that global fibrations by calibrated submanifolds are highly constrained and typically forced to develop singularities.

This phenomenon is already visible in $G_2$-geometry. In the study of semi-flat $G_2$-manifolds and their coassociative fibrations, Baraglia \cite{Baraglia10} showed that full $G_2$-holonomy manifolds fibering over a $3$-manifold necessarily admit singular fibers, providing strong evidence that singularities are an intrinsic feature of fibrations of manifolds with full exceptional holonomy. Motivated by this picture, the following $\m{Spin}(7)$-version was conjectured \cite[p. 1]{englebert2024} that the torsion-free $\m{Spin}(7)$-manifold with full holonomy does not admit a smooth Cayley fibration. Our main result is a proof of this conjecture.

\begin{manualtheorem}{A}\label{main theorem}
    Suppose that $(X, \Phi)$ is a closed full-holonomy $\m{Spin}(7)$-manifold. Then $X$ does not admit a Cayley fibration without singular fibers.
\end{manualtheorem}

Our proof proceeds by contradiction, exploiting the rigid geometric structure of Cayley fibrations. Using fiberwise integration and the negative-definiteness of the $\mathrm{Spin}(7)$-intersection form $Q_\Phi(\alpha,\beta)=\langle\alpha\cup \beta\cup[\Phi],[X]\rangle$ on $H^2(X; \mathbb{R})$, we show that the fibre is simply connected and the base space is negative-definite. Combining this with characteristic-class identities for Cayley fibers, the multiplicativity of Euler characteristic and signature, and the relation between these topological invariants on a full-holonomy $\mathrm{Spin}(7)$-manifold, we arrive at a Diophantine equation \eqref{Diophantic},
\begin{equation*}
    \Big(b^2_-(B) - 2\Big) \cdot \Big(b^2_+(F)  + 1 \Big) = 8.
\end{equation*}
Together with the topology of simply connected smooth $4$ manifolds, this reduces the possible homeomorphism types of the hypothetical nonsingular Cayley fibration to two cases: either the fiber is homeomorphic to $E(2)\cong K3$ and the base to $\#^4\overline{\mathbb{CP}}^2$, or the fiber is homeomorphic to $E(4)\cong K3\#K3\#(S^2\times S^2)$ and the base to $\#^3\overline{\mathbb{CP}}^2$. The two cases are ruled out by gauge-theoretic arguments implying the following spinnability criterion for four dimensional families of manifolds of homeomorphism type $E(2)$ and $E(4)$.

\begin{manualtheorem}{B}
\label{mainthm: X non spin}
Let $n\in\{1,2\}$ and suppose that $\pi\colon X\to B$ is a smooth fiber bundle that satisfies the following conditions:\vspace{-0.3cm}
\begin{itemize}
    \item The base $B$ is a closed, simply-connected smooth 4-manifold;
    \item The fiber $F$ is a smooth manifold homeomorphic to the elliptic surface $E(2n)$.
\end{itemize}\vspace{-0.3cm}
Let $V\pi\coloneqq \ker(T\pi: TX\to TB)$ be the vertical vector bundle. It holds that 
\begin{equation*}
    w_{2}(V\pi)^n=0\in H^{2n}(X,\mathbb{Z}_2).
\end{equation*}
\end{manualtheorem}

The case $n=1$ is a straightforward consequence of a result proved by Baraglia--Konno \cite[Corollary 1.3]{Baraglia_Konno22}. See also Kronheimer--Mrowka \cite[Proposition 2.1]{KronheimerMrowka2020} for the explicit statement over $S^2$ for the diffeomorphism class of $K3$. The result for an arbitrary simply-connected base follows by a standard pullback argument. 

The case $n=2$ requires an extension of these ideas. We argue by contradiction and assume that a counterexample $\pi:X\to B$ exists. We consider a finite-dimensional approximation of the family Seiberg--Witten monopole map
\begin{equation*}
    \mathsf{SW}^{\mathrm{app}}_B\colon
    W^+\oplus V^+
    \longrightarrow
    W^-\oplus V^-.
\end{equation*}
Here $V^\pm$ are trivial real vector bundles over $B$ satisfying
\begin{equation*}
    \operatorname{rank}(V^-)-\operatorname{rank}(V^+) = b^+(E(4)) = 7,
\end{equation*}
while $W^\pm$ are complex vector bundles whose virtual difference $[W^+]-[W^-]\in K(B)$ is the index of the family Dirac operator. Compactness of the Seiberg--Witten equations then gives a pointed map between the fiberwise one-point compactifications,
\begin{equation*}
    \mathsf{SW}^{\mathrm{app}}_{B,+}
    \colon
    S_B^{W^+\oplus V^+}
    \longrightarrow
    S_B^{W^-\oplus V^-}.
\end{equation*}

The key point is to retain the $\mathrm{Pin}(2)$-symmetry of the monopole map and eventually show that the resulting equivariant map cannot exist. A new difficulty arises because, under the assumption $w_2(V\pi)^2\neq 0$, the vertical tangent bundle $V\pi$ does not admit a global spin structure. Consequently, the usual fiberwise $\mathrm{Pin}(2)$-symmetries do not assemble into a global action of a fixed copy of $\mathrm{Pin}(2)$ on the spaces $S_B^{W^\pm\oplus V^\pm}$.

To overcome this, we introduce the notion of a $\mathcal{L}$-twisted family spin structure, where $\mathcal{L}$ is a complex line bundle over $B$ that satisfies \[w_{2}(V\pi)=\pi^*c_{1}(\mathcal{L}) \pmod2.\] The family $\pi:X\to B$ carries such a structure, and the corresponding symmetries assemble into a group bundle $G \to B$ with fiber $\mathrm{Pin}(2)$, giving fiberwise actions
\begin{equation*}
    G \times_B S_B^{W^\pm\oplus V^\pm}
    \longrightarrow
    S_B^{W^\pm\oplus V^\pm}.
\end{equation*}
Thus $\mathsf{SW}^{\mathrm{app}}_{B,+}$ may be regarded as a morphism in $B\backslash \cat{Top}^{G}/B$, the category of pointed spaces over $B$ equipped with these twisted $\mathrm{Pin}(2)$-actions.

We then apply a desuspension argument to reduce to the case $W^-=B\times\{0\}$. Passing to $S^1$-orbit spaces removes the twisting in the residual symmetry and produces an equivariant map between spaces carrying genuine global $C_2$-actions. These $C_2$-actions encode information about the twisting line bundle $\mathcal{L}$ and the class $w_{2}(V\pi)$. Finally, we show using $C_2$-equivariant $K$-theory, together with Atiyah--Hirzebruch spectral sequence calculations, that the required equivariant map cannot exist, yielding the desired contradiction.

We note that several of the constructions and techniques developed here (e.g. twisted family spin structures, the Bauer--Furuta invariant as a parametrized equivariant homotopy class, desuspension techniques, and the reduction to a $C_2$-equivariant problem) apply to general families of spin 4-manifolds. We expect these methods to have further applications to the study of the topology of $\operatorname{BDiff}(F)$ for a spin 4-manifold $F$.

\textbf{Acknowledgments:} The authors thank Thomas Walpuski for insightful discussions and comments on a draft of this article, Johannes Nordström for valuable comments on an earlier version of this article, and David Baraglia and Scotty Tilton for helpful correspondence. J.L. is thankful to Sander Kupers, who mentioned the idea of defining the family Bauer-Furuta invariant as a parametrized stable homotopy class back in 2023. J.L. and J.R. are grateful to the organizers of the Masterclass: Seiberg--Witten Invariants at the University of Copenhagen in June 2026, where discussions following a talk by J.R. initiated the collaboration that led to the present version of this article. J.R. wishes to acknowledge the support of the Deutsche Forschungsgemeinschaft (DFG, German Research Foundation) under Germany’s Excellence Strategy through the Berlin Mathematics Research Center MATH+ (EXC-2046/2, project ID: 390685689). J.L. is partially supported by  National Key R\&D Program of China 2025YFA1017500 and NSFC 12271281. 

\section{Spin(7)-Background}
\label{Spin(7)-background}

Let $(W,g)$ be an eight dimensional Euclidean vector space. The group $\m{Spin}(W,g)\cong \mathrm{Spin}(8)$ has three real eight-dimensional representations: the vector representation on the Euclidean space $(W,g)$, and two chiral spinor representations $\mathbb{S}^+_g$ and $\mathbb{S}^-_g$. Both $\mathbb{S}^+_g$ and $\mathbb{S}^-_g$ can be endowed with $\mathrm{Spin}(8)$-invariant inner products, such that the Clifford multiplication 
\begin{equation*}
    \mathrm{cl}_g\colon W\rightarrow \mathrm{End}(\mathbb{S}^+_g,\mathbb{S}^-_g)\oplus \mathrm{End}(\mathbb{S}^-_g,\mathbb{S}^+_g)
\end{equation*}
is skew-adjoint. Given a unit length positive spinor $\psi_g\in \mathbb{S}^+$, we can identify both 
\begin{equation*}
    \mathrm{cl}_g(\bullet)\psi_g\colon W\cong \mathbb{S}^-_g \und{1.0cm}\m{Stab}(\psi_g)\cong \m{Spin}(7).
\end{equation*}
Since $\m{Stab}(\psi_g)$ fixes the Euclidean structure on $(W,g)$ we deduce that $\m{Spin}(7)\subset \m{SO}(8)$. Using $\psi_g$ we can define the \textbf{Cayley form}
\begin{equation*}
    \Phi\coloneqq \left<\psi_g,\m{cl}^4_g\circ \psi_g\right>\in \wedge^4 W^*.
\end{equation*}
The stabilizer of $\Phi$ under the $\m{GL}(W)$ action on $\wedge^4 W^*$ agrees with $\m{Spin}(7)$ and (up to a sign) we can recover $\psi_g$ and $g$ from the associated Cayley form.

The exterior algebra $\wedge^\bullet W^*$ decomposes into the following $\m{Spin}(7)$-representations 
\begin{align*}
    \wedge^0W^*\cong \wedge^0_1,\hspace{1.0cm}\wedge^1W^*\cong \wedge^1_{\Phi,8}&\,\hspace{1.0cm}\wedge^2W^*\cong \wedge^2_{\Phi,21}\oplus \wedge^2_{\Phi,7},\hspace{1.0cm}\wedge^3W^*\cong \wedge^3_{\Phi,8}\oplus \wedge^3_{\Phi,48}\\
    \und{1.0cm}&\wedge^4W^*\cong \wedge^4_{\Phi,1}\oplus \wedge^4_{\Phi,7}\oplus \wedge^4_{\Phi,27}\oplus \wedge^4_{\Phi,35}
\end{align*}
indexed by their dimension, such that $\wedge^k_{\Phi,\rho}\cong \wedge^{k'}_{\Phi,\rho}$ and moreover $\wedge^kW^*\cong \wedge^{8-k}W^*$.

The Cayley form is an example of a (linear) calibration and satisfies the calibration equality 
\begin{equation*}
    |P|^2=\Phi(P)^2+|\pi_{\psi_g^\perp}(\m{cl}_g(P))\psi_g|^2
\end{equation*}
for all oriented four planes $P$. Such a four plane will be referred to as a \textbf{Cayley plane} if 
\begin{equation*}
    \Phi(P)=1\hspace{1.0cm}\text{or equivalently}\hspace{1.0cm}\m{cl}_g(P)\psi_g=\psi_g.
\end{equation*}

\begin{rem}
\label{cayleyplanestab}
    In particular, a four plane $P$ is Cayley if and only if $P^\perp$ is Cayley and the stabilizer of a Cayley plane decomposition of $W\cong P\oplus P^\perp$ is isomorphic to $\m{Sp}(1)_{\top}\m{Sp}(1)_+\m{Sp}(1)_\perp\cong (\m{Sp}(1)_{\top}\times\m{Sp}(1)_+\times\m{Sp}(1)_\perp)/\mathbb{Z}_2 $. Here, the factor $\m{Sp}(1)_\top\m{Sp}(1)_+$ acts as $\m{SO}(P)$ and $\m{Sp}(1)_+\m{Sp}(1)_\perp$ as $\m{SO}(P^\perp)$.
\end{rem}

Let $X$ be a spinnable Riemannian $8$-manifold. We identify the bundle 
\begin{equation*}
    \m{Cay}_+(X)\coloneqq \m{Fr}_+(X)/\m{Spin}(7)
\end{equation*}
with a (codimension 27) subbundle of four forms $\wedge^4 T^*X$. Any smooth section, referred to as a Cayley form on $X$, yields a  $\m{Spin}(7)\hookrightarrow \m{SO}(8)$-reduction 
\begin{equation*}
    \m{Fr}(X,\Phi)\hookrightarrow \m{Fr}_+(X,g_\Phi)
\end{equation*}
of the normed oriented frame bundle associated to the Riemannian metric $(X,g_\Phi)$, by pulling back the bundle $\m{Fr}_+(X)\rightarrow \m{Cay}_+(X)$. If $\Phi$ is parallel with respect to the Levi-Civita connection $\nabla^{g_\Phi}$, then by the holonomy principle $\m{Hol}(g_\Phi)\subset \m{Spin}(7)$. The holonomy group $\m{Spin}(7)$ appears in the Berger-list \cite{berger} of possible holonomy groups, together with the group $G_2$, the two exceptional classes. Their existence has been proven locally \cite{bryant1987} and later in the compact case \cite{Joyce1996b}. 

\begin{theorem}\cite{Joyce1996b}
    Let $(X,\Phi)$ be a torsion-free $\m{Spin}(7)$-manifold. Then it is full-holonomy $\m{Spin}(7)$, i.e. $\m{Hol}(g_\Phi)=\m{Spin}(7)$ if and only if
    \begin{equation*}
        \pi_1(X)=1\und{1.0cm}\hat{A}(X)=1.
    \end{equation*}
    Both imply that $b^1(X)=b^1_{\Phi,8}(X)=0$ and $b^2_{\Phi,7}(X)=0$. In this case, the $\Phi$-intersection pairing 
    \begin{equation*}
        \left<\bullet,\bullet\right>_\Phi\colon H^2(X;\mathbb{R})^{\otimes 2}\rightarrow \mathbb{R},\,(\alpha\otimes\beta)\mapsto \left<\alpha\cup\beta\cup\Phi\right>
    \end{equation*}
    is negative-definite.
\end{theorem}

Fernández \cite{Fernandez1986} proved that the property of being parallel with the respect to itself is equivalent to the four form being closed. Consequently, torsion free $\m{Spin}(7)$-structures are examples of calibrated structures in the sense of \cite{Harvey1982}. An embedded four dimensional submanifold $\iota\colon C\hookrightarrow (X,\Phi)$ is called calibrated, or \textbf{Cayley submanifold}, if 
\begin{equation*}
    \iota^*\Phi=\m{vol}_{\iota^*g}\hspace{1.0cm}\text{or equivalently}\hspace{1.0cm}\m{cl}_{g}(\m{vol}_{\iota^*g})\iota^*\psi=\iota^*\psi.
\end{equation*}
Moreover, the calibration property implies that Cayley submanifolds are always volume minimizing in their respective homology classes. The deformation theory of Cayley submanifolds is elliptic \cite[Section 6]{mclean1998Deformations} and the virtual dimension of their moduli space equals $v\dim(\mathcal{C}ay(C,X,\Phi))=\frac{\sigma(C)+\chi(C)}{2}-\iota_*[C]\cdot \iota_*[C]$.

\begin{definition}
    Let $(X,\Phi)$ be a torsion-free $\m{Spin}(7)$-manifold, assume that there exists a smooth map $\pi\colon X\rightarrow B$ to a compact four manifold such that the fibers $\iota_b\colon \pi^{-1}(b)\hookrightarrow X$ are all Cayley submanifolds. Then we refer to $(X,\Phi,\pi)$ as a \textbf{Cayley fibration}.
\end{definition}

\section{Restricting the Topological Type of Smooth Cayley Fibrations}
\label{Restricting the Topological Type of Smooth Cayley Fibrations}

Let $\pi\colon (X, \Phi) \to B$ be a smooth Cayley fibration with $\m{Hol}(g_\Phi) = \m{Spin}(7)$. In this section we use the negativity of the $\Phi$-intersection form, characteristic class identities and the multiplicativity of Euler characteristic and signature to show that only two topological possibilities can occur. 

Henceforth, let $F\cong \pi^{-1}(b)$ be the fiber of $\pi$ and $\iota_b\colon F\hookrightarrow X$ the embedding of the fiber over $b\in B$.

Notice that since $\pi_1(X) = \pi_0(X) = 1$ the following holds. Suppose that the base space $B$ is not simply connected. Then if $p\colon  \tilde{B} \to B$ is its universal cover, we have a unique lift of $\pi$ along $p$
\begin{equation*}
\begin{tikzcd}
	& {\tilde{B}} \\
	X & B
	\arrow["p"{description}, from=1-2, to=2-2]
	\arrow["{\tilde{\pi}}"{description}, dashed, from=2-1, to=1-2]
	\arrow["\pi"{description}, from=2-1, to=2-2]
\end{tikzcd}
\end{equation*}
which follows from the lifting property for the covering maps \cite[Proposition 3.5]{May99}. This is justified since we trivially have $\pi_*(\pi_1(X)) = 1$. The smoothness of $\tilde{\pi}$ is immediate, i.e. after replacing $\pi$ with $\tilde{\pi}$ it is harmless to assume that the base is simply connected (and thus the fiber is connected as well).

\begin{rem}
\label{samevolume}
    Being a fibration with calibrated fibres implies that all fibers are minimal and have the same volume. (They represent the same homology class.)
\end{rem}

\begin{rem}
        Let $\pi\colon (X,\Phi)\rightarrow B$ be a Cayley fibration. The Riemannian metric associated to $\Phi$ induces an Ehresmann connection 
        \begin{equation*}
            TX\cong H\oplus V\pi.
        \end{equation*}
        Using this splitting, the $\m{Spin}(7)$-frame bundle admits a $\mathrm{Sp}(1)_\top\mathrm{Sp}(1)_+\mathrm{Sp}(1)_\perp$-reduction 
    \begin{equation*}
        Q\hookrightarrow \m{Fr}(X,\Phi).
    \end{equation*}
    to frames preserving the splitting $TX\cong H\oplus V\pi$. Its torsion includes both the second fundamental form of the fibers $\m{II}_g\in \Gamma(X,\m{Sym}^2(V^*\pi,H))$ as well as the curvature $F_H\in \Gamma(X,\wedge^2 H^*\otimes V\pi)$ of the Ehresmann connection.
    
    Using the reduction $Q\hookrightarrow\m{Fr}(X,\Phi)$ we can identify the bundles
    \begin{align*}
        V^*\pi\cong& Q\times_{\mathrm{Sp}(1)_\top\mathrm{Sp}(1)_+\mathrm{Sp}(1)_\perp}V_{1,1,0},\\
        H^*\cong& Q\times_{\mathrm{Sp}(1)_\top\mathrm{Sp}(1)_+\mathrm{Sp}(1)_\perp}V_{0,1,1}\cong \pi^*T^*B \und{1.0cm}\\
        \wedge^2_+V^*\pi\cong& Q\times_{\mathrm{Sp}(1)_\top\mathrm{Sp}(1)_+\mathrm{Sp}(1)_\perp}V_{0,2,0}\cong  \wedge^2_+H^*.
    \end{align*}
Restricting these bundles to the fibers shows that $\wedge^2_+ V^*\pi_b\cong \wedge^2_+T^*F\cong \wedge^2_+H_b^*\cong \underline{\mathbb{R}}^3_{F}$ by the triviality of the horizontal bundle along the fibers.
\end{rem}

These remarks lead to the following conclusion. Recall that \cite[Proposition 10.6.6]{joyce2000compact} implies that for any $\sigma \in H^2(X;\mathbb{R})$ we have
\begin{equation*}
    \left<\sigma,\sigma\right>_\Phi=\int_X \sigma \wedge\sigma\wedge\Phi \leq 0.
\end{equation*}
Now, for any $\omega\in H^2(B;\mathbb{R})$, fiber integration and Remark \ref{samevolume} give
\begin{equation*}
    \left< \pi^*\omega,\pi^*\omega\right>_\Phi= \int_X \pi^*\omega\wedge\pi^*\omega\wedge\Phi= \int_B \omega\wedge\omega\wedge\left(\int_{X/B}\Phi\right) = \m{vol}(F)\int_B\omega\wedge\omega
     = \m{vol}(F)\,Q_B(\omega,\omega).
\end{equation*}
Since the $\Phi$-intersection form on $H^2(X;\mathbb{R})$ is negative-definite, it follows that $Q_B(\omega,\omega)\leq 0$ for every $\omega\in H^2(B;\mathbb{R})$. Thus, the intersection form of $B$ is negative-semidefinite.  On the other hand, $B$ is a closed oriented four-manifold, so its intersection form is nondegenerate by Poincar\'e
duality. Consequently, the intersection form of $B$ is negative-definite.

We next show that the pullback map
\begin{equation*}
    \pi^*\colon H^2(B;\mathbb{R})\longrightarrow H^2(X;\mathbb{R})
\end{equation*}
is injective. Indeed, if $\pi^*\omega=0$, then the computation above gives
\begin{equation*}
    0
    = \left< \pi^*\omega,\pi^*\omega\right>_\Phi
    = \m{vol}(F)\,Q_B(\omega,\omega).
\end{equation*}
Since $\m{vol}(F)>0$ and $Q_B$ is negative-definite, we conclude that $\omega=0$.

Finally, since $B$ is simply-connected, the local coefficient system in the Leray--Serre spectral sequence is trivial. Its low-degree exact
sequence contains
\begin{equation*}
    0=H^1(X;\mathbb{R})
    \longrightarrow H^1(F;\mathbb{R})
    \xrightarrow{d_2} H^2(B;\mathbb{R})
    \xrightarrow{\pi^*} H^2(X;\mathbb{R}).
\end{equation*}
Exactness and the injectivity of $\pi^*$ imply that $d_2=0$, while exactness at $H^1(F;\mathbb{R})$ implies that $d_2$ is injective. Therefore $H^1(F;\mathbb{R})=0$.

To summarize the discussion, we have proved the following.

\begin{prop}\label{topology of the fiber}
    The intersection form of the base space $B$ is negative-definite and $b^1(F) = 0$.
\end{prop}

From the fibration structure it follows that $N\iota_b \cong \underline{\mathbb{R}}^4$ so that the bundle $\wedge^2_{+,b}N^*\iota_b \cong \wedge^2_{+}T^*F_b$ is trivial \footnote{See \cite[Remark 9.5]{salamon2010notes} for the proof that there is such a canonical identification induced by $\Phi$}. This means \cite[Exercise 10.1.3]{GompfStipsicz} that $2\chi(F) + 3 \sigma(F) = 0$ and $w_2(F) = 0$; in particular $2(2+b^2_+(F)+b^2_-(F)) + 3 (b^2_+(F)-b^2_-(F))= 0$, so that 
\begin{equation*}
    b^2_-(F) = 4 + 5 b^2_+(F) \quad \textnormal{and} \quad \sigma(F)= -4(1+ b^2_+(F) ).
\end{equation*}
Now, since the Euler characteristic of a fibration is multiplicative, it follows that
\begin{equation}\label{chi computation}
    \chi(X) = \chi(B)\chi(F) = (2+b^2_-(B))\cdot 6(1+b_+^2(F)).
\end{equation}
Furthermore, a theorem of Chern-Hirzebruch-Serre \cite{chernhirzebruchserre57} states that the signature of a fiber bundle with a simply-connected base has a multiplicative signature, so that
\begin{equation}\label{sigma computation}
    \sigma(X) = \sigma(B)\sigma(F) = 4 b^2_-(B) \cdot (b^2_+(F) + 1).
\end{equation}
Now, \cite[pp. 259]{joyce2000compact} gives
\begin{equation*}
    24\hat{A}(X) = -1 + b^1(X)-b^2(X)+b^3(X)+b^4_+(X)-2b^4_-(X) 
\end{equation*}
and by Poincaré duality
\begin{equation*}
    \chi(X) = 2-2b^1(X)+2b^2(X)-2b^3(X)+b_4^+ (X) + b_4^-(X)
\end{equation*}
so that
\begin{equation}
    \chi(X)+48 \hat{A}(X) = 3(b^+_4(X) - b^-_4(X)) = 3 \sigma(X).
\end{equation}
We note for completeness that this equality already appears in \cite[Equation (2)]{crowleynordstrom2015new}. Since for a closed torsion-free $\mathrm{Spin}(7)$ manifold it holds that $\hat{A}(X) = 1$, so that substituting in \eqref{chi computation} and \eqref{sigma computation}, we get 
\begin{equation}
\label{A-hat-chi-sigma relation}
    6\cdot \Big(2+b^2_-(B)\Big) \cdot \Big(1+b_+^2(F)\Big)+48 =12 \cdot b^2_-(B) \cdot \Big( 1 + b^2_+(F) \Big).
\end{equation}
Rearranging this formula yields
\begin{equation}
\label{Diophantic}
    \Big(b^2_-(B) - 2\Big) \cdot \Big(b^2_+(F)  + 1 \Big) = 8.
\end{equation}
This narrows down the possibilities to 
\begin{equation}\label{topological types}
    (b^2_-(B), b^2_+(F)) \in \{ (4, 3), (3, 7)\}.
\end{equation}

Let us remark that $b^2_+(F)=0$ and $b^2_+(F) = 1$ are not valid possibilities since $F$ is spin and in these cases the respective equalities $\sigma(F) = -4$ and $\sigma(F) = -8$ contradict Rokhlin's theorem which states that $16|\sigma(F)$.

Before we consider \eqref{topological types} case-by-case, let us first determine the homeomorphism type of the base and fiber. Recall that Freedman's classification theorem \cite{Freedman82} states that the homeomorphism type of a simply-connected $4$-manifold is determined by its intersection form and its Kirby--Siebenmann invariant\footnote{Recall that the Kirby--Siebenmann invariant vanishes if the manifold admits a smooth structure}. 

First, let us consider the base. We have already shown that $B$ is simply-connected and that its intersection form is negative definite. Since $B$ is smooth, Donaldson's diagonalisation theorem \cite{Donaldson83} implies that $Q_B \cong-I_{b^2_-(B)}$, so we can immediately conclude 
\begin{equation*}
    B \cong_{\cat{Top}} \#^{\,b^2_-(B)} \overline{\mathbb{CP}}^2.
\end{equation*}
Note that in particular, $B$ is not spin. 

Now, we will explain the homeomorphism type of the fiber. Let
\begin{equation*}
\partial\colon\pi_2(B)\longrightarrow \pi_1(F)
\end{equation*}
be the connecting homomorphism in the long exact sequence of homotopy groups associated with the fibration. Since $\pi_1(X)=0$, the map $\partial$ is surjective. Furthermore, the image of the connecting homomorphism of any fibration is contained in the first Gottlieb group $G_1(F)$ \cite[Section 1, Property (7)]{Gottlieb89}. We shall not need the definition of $G_1(F)$, only the following consequence of Gottlieb's work: if a space has the homotopy type of a compact connected polyhedron and nonzero Euler characteristic, then its first Gottlieb group is trivial \cite[Theorem IV.1]{Gottlieb65}. These assumptions apply here, since $F$ is a closed connected manifold and $\chi(F)=6(1+b_2^+(F))>0$. Thus $G_1(F)=0$, and hence
\begin{equation*}
    \pi_1(F)=0.
\end{equation*}
Since the manifold $F$ is simply-connected and spin, the intersection form $Q_F$ is even and unimodular. Moreover, in both cases under consideration the form is indefinite. By the classification of indefinite even unimodular forms \cite[Theorem 1.2.14]{GompfStipsicz}, $Q_F$ is determined by $b^2_+(F)$ and $b^2_-(F)$. Consequently, depending on the case in \eqref{topological types}, it is either $2(-E_8)\oplus 3H $ or $4(-E_8)\oplus 7H$. Once again, by Freedman's classification, these are realized uniquely up to homeomorphism by $K3$ and $K3\# K3 \# (S^2 \times S^2)$, respectively.

Furthermore, computing the Leray-Serre spectral sequence allows us to determine the Betti numbers of $X$ in each of the cases \eqref{topological types}.

Summing up the discussion above, we are left with the following two possibilities. 

\begin{theorem}\label{classification theorem}
    If $\pi\colon (X,\Phi) \to B$ is a smooth Cayley fibration without singular fibers, then one of the following cases must hold true:
\begin{enumerate}[
    label=\textbf{Case \arabic*:},
    leftmargin=*,
    align=left
]
\item
\begin{itemize}[leftmargin=2em]
    \item $(b^2(X), b^3(X), b^4_+(X), b^4_-(X)) = (26, 0, 77, 13)$,
    \item the base $B$ is homeomorphic to $\#^4 \overline{\mathbb{CP}}^2$,
    \item the fiber $F$ is homeomorphic to $K3$.
\end{itemize}

\item
\begin{itemize}[leftmargin=2em]
    \item $(b^2(X), b^3(X), b^4_+(X), b^4_-(X)) = (49, 0, 118, 22)$,
    \item the base $B$ is homeomorphic to $\#^3 \overline{\mathbb{CP}}^2$,
    \item the fiber $F$ is homeomorphic to $K3 \# K3 \# (S^2 \times S^2)$.
\end{itemize}
\end{enumerate}
\end{theorem}

\section{Proof of Theorem \ref{main theorem}}
\label{Proof of Theorem A}

We now combine the topological classification in Theorem
\ref{classification theorem} with Theorem \ref{mainthm: X non spin} to prove Theorem \ref{main theorem}. We first record the cohomological fact that will allow us to detect characteristic classes pulled back from the
base.

\begin{lemma}
    \label{lem:pullback-injective-degrees-two-and-four}
    Let $\pi\colon X\rightarrow B$ be a fiber bundle whose base and fiber are closed, simply-connected $4$-manifolds. Then the pullback maps
    \begin{equation*}
        \pi^*\colon H^k(B;\mathbb{Z}_2)
        \rightarrow H^k(X;\mathbb{Z}_2)
    \end{equation*}
    are injective for $k=2$ and $k=4$.
\end{lemma}

\begin{proof}
    Consider the cohomological Leray--Serre spectral sequence with $\mathbb{Z}_2$-coefficients. Since $B$ is simply-connected, all local coefficient systems are trivial, and therefore
    \begin{equation*}
        E_2^{p,q}
        \cong
        H^p(B;\mathbb{Z}_2)\otimes H^q(F;\mathbb{Z}_2).
    \end{equation*}
    The edge homomorphism associated with the bottom row is the pullback $\pi^*$. No differential leaves this row. The only differential that could enter bidegree $(2,0)$ is
    \begin{equation*}
        d_2\colon E_2^{0,1}\rightarrow E_2^{2,0},
    \end{equation*}
    whose source vanishes because $F$ is simply-connected.

    The only possible differentials entering bidegree $(4,0)$ have sources
    \begin{equation*}
        E_2^{2,1},
        \qquad
        E_3^{1,2}
        \und{1.0cm}
        E_4^{0,3}.
    \end{equation*}
    These groups vanish because, respectively,
    \begin{equation*}
        H^1(F;\mathbb{Z}_2)=0,
        \qquad
        H^1(B;\mathbb{Z}_2)=0
        \und{1.0cm}
        H^3(F;\mathbb{Z}_2)=0,
    \end{equation*}
    where the last equality follows from Poincar\'e duality. It follows that
    \begin{equation*}
        E_\infty^{k,0}=E_2^{k,0}=H^k(B;\mathbb{Z}_2)
        \qquad\text{for }k\in\{2,4\}.
    \end{equation*}
    In total degree $k$, the term $E_\infty^{k,0}$ identifies with the filtration subgroup $F^kH^k(X;\mathbb{Z}_2)$. Under this identification, its inclusion into $H^k(X;\mathbb{Z}_2)$ is the edge homomorphism $\pi^*$. Hence $\pi^*$ is injective in degrees $2$ and $4$.
\end{proof}

\begin{proof}[Proof of Theorem \ref{main theorem}]
    Suppose, for a contradiction, that a closed full-holonomy $\m{Spin}(7)$-manifold $(X,\Phi)$ admits a smooth Cayley fibration
    \begin{equation*}
        \pi\colon (X,\Phi)\rightarrow B.
    \end{equation*}
    The $\m{Spin}(7)$-structure makes $X$ spin. The Cayley calibration orients the fibers, and hence the vertical tangent bundle $V\pi$. Together with
    the orientation of $X$, this also orients the horizontal bundle. An Ehresmann connection gives a splitting
    \begin{equation*}
        TX\cong V\pi\oplus\pi^*TB.
    \end{equation*}
    Since both summands are oriented, the degree-two part of the Whitney product formula gives
    \begin{align}
        \label{eq:vertical-w2-pullback}
        0=w_2(TX)=w_2(V\pi)+\pi^*w_2(TB),
    \end{align}
    and hence
    \begin{equation*}
        w_2(V\pi)=\pi^*w_2(TB).
    \end{equation*}

    In Case $1$ of Theorem \ref{classification theorem}, the fiber is homeomorphic to $K3=E(2)$. Theorem \ref{mainthm: X non spin} with $n=1$
    implies that $w_2(V\pi)=0$. Equation \eqref{eq:vertical-w2-pullback} and Lemma
    \ref{lem:pullback-injective-degrees-two-and-four} then imply
    \begin{equation*}
        w_2(TB)=0.
    \end{equation*}
    On the other hand, $B$ is homeomorphic to
    $\#^4\overline{\mathbb{CP}}^2$, so its intersection form is $-I_4$ and is therefore odd. The Wu formula then implies that $w_2(TB)\neq 0$, which is
    a contradiction.\\

    In Case $2$, the fiber is homeomorphic to
    \begin{equation*}
        K3\mathbin{\#}K3\mathbin{\#}(S^2\times S^2),
    \end{equation*}
    and hence to $E(4)$. Applying Theorem
    \ref{mainthm: X non spin} with $n=2$ and using
    \eqref{eq:vertical-w2-pullback}, we obtain
    \begin{equation*}
        0=w_2(V\pi)^2=\pi^*\bigl(w_2(TB)^2\bigr).
    \end{equation*}
    Lemma \ref{lem:pullback-injective-degrees-two-and-four} therefore gives $w_2(TB)^2=0$. On the other hand, the mod-$2$ reduction of the first Pontryagin class satisfies
    \begin{equation*}
        p_1(TB)\equiv w_2(TB)^2\pmod{2}.
    \end{equation*}
    Since $B$ has negative-definite intersection form of rank $3$, one has $\sigma(B)=-3$. The Hirzebruch signature theorem consequently gives
    \begin{align*}
        \left\langle w_2(TB)^2,[B]\right\rangle
        &\equiv
        \left\langle p_1(TB),[B]\right\rangle
        \pmod{2}
        \\
        &=3\sigma(B)
        \pmod{2}
        \\
        &\equiv 1
        \pmod{2}.
    \end{align*}
    Thus $w_2(TB)^2\neq 0$, which is again a contradiction.\\

    Both alternatives in Theorem \ref{classification theorem} are impossible. Hence a closed full-holonomy $\m{Spin}(7)$-manifold cannot admit a smooth Cayley fibration without singular fibers.
\end{proof}

\section{Proof of Theorem \ref{mainthm: X non spin}}
\label{Proof of Theorem B}

The proof of Theorem \ref{mainthm: X non spin} uses the families Bauer--Furuta invariant. The case $n=1$ follows directly from a result of Baraglia, whereas the case $n=2$ requires a family version of the $\m{Pin}(2)$-equivariant monopole map. Starting from a twisted family spin structure, we construct a finite-dimensional approximation of this map and use a desuspension argument to obtain a $C_2$-equivariant null-homotopy. We then show, using equivariant $K$-theory, that this null-homotopy cannot exist.

\subsection{The case $n=1$}
As we mentioned earlier, the case $n=1$ is a straightforward consequence of the following proposition of Baraglia. 

\begin{prop}{\cite[Proposition 7.6]{Baraglia23a}} \label{Baraglia proposition}
    Let $\pi\colon  X \to B$ be a smooth fiber bundle with a fiber homeomorphic to a $K3$ surface. Then $w_2(V\pi) = \pi^*(w_2(H^+))$, where $H^+ \to B$ denotes the bundle whose fiber over $b$ is a maximal positive definite subspace of $H^2(E_b;\mathbb{R})$.
\end{prop}

\begin{proof}[Proof of Theorem \ref{mainthm: X non spin} for $n=1$]
    The flat cohomology bundle $H^2(X/B;\mathbb{R}) \to B$ is classified by the monodromy representation of $\pi_1(B)$ which is trivial by assumption. Now, this means that the bundle $H^+ \to B$ can be identified with a map $B \to \textnormal{Gr}^+_3(\mathbb{R}^{3, 19})$. But this Grassmannian is contractible \cite[p. 1735-1736]{Baraglia23b}, i.e. such a map is nullhomotopic, and thus $H^+$ is trivial. The case $n=1$ in Theorem \ref{mainthm: X non spin} now follows by Proposition \ref{Baraglia proposition}.
\end{proof}

\subsection{Cohomological reduction for $n=2$}
\label{Cohomological reduction for n=2}
We now turn to the case $n=2$. Throughout the remainder of this section, let $\pi\colon X\to B$ be a smooth bundle whose fiber $F$ is homeomorphic to $E(4)$. We argue by contradiction and assume that
\begin{equation*}
    w_{2}(V\pi)^2\neq 0 \in H^4(X; \mathbb{Z}_2).
\end{equation*}
We will first reduce the topology of the base, then construct a twisted family spin structure and an associated finite-dimensional approximation of the family monopole map. Finally, we will derive a contradiction using $C_2$-equivariant $K$-theory.

Let $i\colon F\hookrightarrow X$ denote the inclusion of a fiber. Then $i^*(w_2(V\pi))=w_{2}(TF)=0$ because $F$ is spin. The Leray--Serre spectral sequence gives an exact sequence
\begin{equation*}
    0\to H^{2}(B;\mathbb{Z}_2)\xrightarrow{\pi^*} H^2(X;\mathbb{Z}_2)\xrightarrow{ i^*}H^2(F;\mathbb{Z}_2).
\end{equation*}
Hence there exists $\overline{a} \in H^2(B; \mathbb{Z}_2)$ such that $w_{2}(V\pi)=\pi^*(\overline{a})$. It follows that $0\neq w_2(V \pi)^2 = \pi^* (\overline{a}^2),$ and therefore $\overline{a}^2 \neq 0$.

Since $B$ is closed and simply connected, $H^3(B; \mathbb{Z}) =0$, and thus $\overline{a}$ admits an integral lift  $a\in H^{2}(B;\mathbb{Z})$. We then have 
\begin{equation*}
    \langle a^2, [B] \rangle = \langle \overline{a}^2, [B]_{\mathbb{Z}_2} \rangle = 1 \pmod 2,
\end{equation*}
so the intersection form of $B$ is odd. By Donaldson's diagonalization theorem and the classification of odd, indefinite symmetric unimodular forms, the intersection form of $B$ is $k\langle1\rangle\oplus l\langle -1\rangle$ for some $k,l$.

By the Whitehead theorem, there is a homotopy equivalence $f\colon (\#^{k}\mathbb{CP}^2)\#(\#^{l}\overline{\mathbb{CP}^2})\to B$ which we can take to be smooth.\footnote{Actually, we can pick $f$ to be homotopic to a homeomorphism by Freedman's theorem \cite{Freedman82}.} Replacing $\pi$ by $f^* \pi$, we may therefore assume that $B = (\#^{k}\mathbb{CP}^2)\#(\#^{l}\overline{\mathbb{CP}^2})$. Let $a_{1},\cdots,a_{k+l}\in H^{2}(B;\mathbb{Z})$ denote the standard basis of $H^2(B;\mathbb{Z})$. After changing the integral lift $a$ by twice an integral class, we may assume that 
\begin{equation*}
    a=\sum_{j=1}^{k+l}s_{j}a_{j} \qquad \text{with } s_{j}\in \{0,1\}
\end{equation*}
with $s_1 \neq 0$.

\subsection{Twisted family spin structures}
\label{Twisted family spin structures}

In the ordinary Bauer--Furuta construction, the monopole map is regarded as a pointed $\m{Pin}(2)$-equivariant map between Hilbert spaces, and its finite-dimensional approximation determines a stable equivariant homotopy class. In the present
situation, however, all the relevant objects vary in families over $B$: the
configuration spaces, the family Dirac operators, their finite-dimensional
approximations, and the resulting representation spheres are bundles over $B$.
We therefore need a parametrized version of the usual pointed equivariant
formalism.

There is a second complication. Although every fiber of $\pi\colon X\rightarrow B$ is spin, the vertical tangent bundle $V\pi$ need not admit a spin structure globally over $X$. The preceding subsection shows that its obstruction is pulled back from the base:
\begin{equation*}
    w_2(V\pi)=\pi^*\overline{a}
\end{equation*}
for an integral class $a\in H^2(B;\mathbb Z)$. Consequently, the spinor bundles
possess their usual quaternionic structures locally over $B$, but these local
structures need not glue to an action of a single fixed copy of $\mathbb H$, or
of $\m{Pin}(2)$. Instead, they glue to bundles of quaternion algebras and
symmetry groups over $B$.

We begin by making the required parametrized language precise.

\begin{definition}
    Let $B\backslash\cat{Top}/B$ denote the category of \textbf{retractive spaces over $B$}. Its objects are diagrams
    \begin{equation*}
        B\xrightarrow{s}E\xrightarrow{p}B
    \end{equation*}
    satisfying $p\circ s=\m{id}_B$. A morphism $f\colon(E,p,s)\longrightarrow(E',p',s')$
    is a continuous map $f\colon E\rightarrow E'$ satisfying $ p'\circ f=p$ and $      f\circ s=s'$. Thus a retractive space over $B$ is precisely a continuously varying family of pointed spaces parametrized by $B$.

    There is a forgetful functor
    \begin{equation*}
        \mathsf{Forget}\colon
        B\backslash\cat{Top}/B\rightarrow
        \cat{Top}/B,\quad
        (E,p,s)\mapsto(E,p).
    \end{equation*}
    Its left adjoint adds a disjoint section. More precisely, if
    $(E,p)\in\cat{Top}/B$, we set
    \begin{equation*}
        E_{+_B}
        \coloneqq
        \bigl(
            E\sqcup B,
            p\sqcup\m{id}_B,
            \iota_B
        \bigr),
    \end{equation*}
    where $\iota_B\colon B\hookrightarrow E\sqcup B$ denotes the inclusion of
    the second summand.
\end{definition}

If $E\rightarrow B$ is a vector bundle, we denote its fiberwise one-point
compactification by $S^E$. The points at infinity define a canonical section,
and hence $S^E$ is naturally an object of $B\backslash\cat{Top}/B$. We continue
to write $S(E)\rightarrow B$ for the ordinary fiberwise sphere bundle.

\begin{definition}
    Let $(E,p_E,s_E)$ and $(F,p_F,s_F)$ be retractive spaces over $B$. Their
    \textbf{fiberwise smash product} is the retractive space
    \begin{equation*}
        E\wedge_BF
        \coloneqq
        \frac{E\times_BF}
        {
            s_E(B)\times_BF
            \,\cup\,
            E\times_Bs_F(B)
        }.
    \end{equation*}
    Fiberwise, this construction satisfies
    \begin{equation*}
        (E\wedge_BF)_b\cong E_b\wedge F_b.
    \end{equation*}
    In particular, for vector bundles $E,F\rightarrow B$, there is a canonical
    identification
    \begin{equation*}
        S^{E\oplus F}\cong S^E\wedge_BS^F.
    \end{equation*}
\end{definition}

We now return to the vertical tangent bundle. Let $a\in H^2(B;\mathbb Z)$ be
the integral lift chosen in the preceding subsection, and fix a Hermitian line
bundle
\begin{equation*}
    \mathcal L\longrightarrow B,
    \qquad
    c_1(\mathcal L)=a.
\end{equation*}
The following definition records that the failure of $V\pi$ to be spin is
entirely induced from the base.

\begin{definition}
    An \textbf{$\mathcal L$-twisted family spin structure} on
    $\pi\colon X\rightarrow B$ consists of a $\m{Spin}^c(4)$-structure
    $\mathfrak s$ on $V\pi$, together with a unitary isomorphism
    \begin{equation*}
        \lambda\colon
        \det(\mathfrak s)
        \xrightarrow{\cong}
        \pi^*\mathcal L.
    \end{equation*}
\end{definition}

The terminology is justified by restricting to a fiber $X_b$. Since
$\pi^*\mathcal L|_{X_b}$ is pulled back from the point $b$, it is trivial on
$X_b$. A choice of unit vector in $\mathcal L_b$ therefore reduces
$\mathfrak s|_{X_b}$ to an ordinary spin structure. Different choices differ
by an element of $S^1$; as $b$ varies, this ambiguity is precisely what
produces the twisted quaternionic structure introduced below.

\begin{prop}
    \label{Existence of an L-twisted family spin structure}
    The vertical tangent bundle $V\pi$ admits an $\mathcal L$-twisted family
    spin structure.
\end{prop}

\begin{proof}
    An oriented vector bundle $E$ admits a $\m{Spin}^c$-structure with
    determinant line bundle $K$ precisely when
    \begin{align}
        \label{Spinc condition}
        w_2(E)=c_1(K)\pmod 2.
    \end{align}
    Indeed, a $\m{Spin}^c$-structure with determinant line bundle $K$ is
    equivalent to a spin structure on $E\oplus K_{\mathbb R}$. Such a structure
    exists precisely when
    \begin{equation*}
        w_2(E\oplus K_{\mathbb R})=0.
    \end{equation*}
    Since
    \begin{equation*}
        w_2(K_{\mathbb R})=c_1(K)\pmod 2,
    \end{equation*}
    this is equivalent to \eqref{Spinc condition}.

    In the present situation,
    \begin{equation*}
        c_1(\pi^*\mathcal L)
        =
        \pi^*a
        \equiv
        \pi^*\overline a
        =
        w_2(V\pi)
        \pmod 2,
    \end{equation*}
    and hence the required $\m{Spin}^c$-structure exists.
\end{proof}

Proposition~\ref{Existence of an L-twisted family spin structure} supplies a
globally defined $\m{Spin}^c$-structure, but not a preferred family of
reductions to $\m{Spin}(4)$. Locally on $B$, a unitary frame of $\mathcal L$
determines such a reduction and hence equips the spinor bundles with their
usual quaternionic structures. On overlaps, these structures are related by
automorphisms of $\mathbb H$. The following bundle of quaternion algebras
packages this gluing intrinsically.

Write
\begin{equation*}
    \mathbb H=\mathbb C\oplus\mathbb Cj,
    \qquad
    j^2=-1,
    \qquad
    jz=\overline z\,j.
\end{equation*}
Consider the homomorphism
\begin{equation*}
    \vartheta\colon S^1
    \rightarrow
    \m{Aut}_{\mathbb R\text{-}\m{Alg}}(\mathbb H),\quad u+vj\mapsto 
    u+zvj.
\end{equation*}
Equivalently, if $z=e^{i\theta}$, then
\begin{equation*}
    \vartheta(e^{i\theta})(h)
    =
    e^{i\theta/2}he^{-i\theta/2}.
\end{equation*}
This expression does not depend on the choice of the square root
$e^{i\theta/2}$.

Let $P_{\mathcal L}\rightarrow B$ denote the unit-frame bundle of
$\mathcal L$. We define the bundle of quaternion algebras
\begin{equation*}
    \mathbb H_{\mathcal L}
    \coloneqq
    P_{\mathcal L}\times_\vartheta\mathbb H
    \longrightarrow B.
\end{equation*}
By the definition of $\vartheta$, its underlying complex vector bundle is
canonically isomorphic to
\begin{equation*}
    \mathbb H_{\mathcal L}
    \cong
    \underline{\mathbb C}_B\oplus\mathcal L.
\end{equation*}

\begin{definition}
    Let $f\colon M\rightarrow B$ be a smooth map and let $E\rightarrow M$ be
    a complex vector bundle. We call $E$ an
    \textbf{$\mathcal L$-twisted quaternionic bundle} if it is a module bundle
    over $f^*\mathbb H_{\mathcal L}$ and the action of the canonical subbundle
    \begin{equation*}
        \underline{\mathbb C}_M
        \subset
        f^*\mathbb H_{\mathcal L}
    \end{equation*}
    agrees with the given complex structure on $E$.
\end{definition}

The family monopole map is defined on Sobolev completions of the vertical
spaces of spinors and differential forms. We must therefore verify that the
twisted quaternionic structure survives passage from a bundle over $M$ to the
corresponding Hilbert bundle of vertical sections over $B$.

\begin{rem}
    Let $q\colon M\rightarrow B$ be a smooth fiber bundle with compact fiber,
    and let $p\colon E\rightarrow M$ be a Hermitian vector bundle. For
    $k\geq0$, denote by
    \begin{equation*}
        W^{k,2}(M/B;E)\longrightarrow B
    \end{equation*}
    the Hilbert bundle of vertical $W^{k,2}$-sections, whose fiber over
    $b\in B$ is
    \begin{equation*}
        W^{k,2}(M/B;E)_b
        \coloneqq
        W^{k,2}(M_b;E|_{M_b}).
    \end{equation*}
    Local trivializations of $q$ and $p$ induce local trivializations of
    $W^{k,2}(M/B;E)$, making it a topological Hilbert bundle over $B$.
\end{rem}

\begin{lemma}
    If $E\rightarrow M$ is an $\mathcal L$-twisted quaternionic bundle
    relative to $q\colon M\rightarrow B$, then
    $W^{k,2}(M/B;E)$ is naturally an $\mathbb H_{\mathcal L}$-module bundle
    for every $k\geq0$.
\end{lemma}

\begin{proof}
    For $b\in B$, the $q^*\mathbb H_{\mathcal L}$-module structure on $E$
    induces an action
    \begin{equation*}
        (h\cdot s)(x)
        \coloneqq
        h\cdot s(x),
        \qquad
        h\in(\mathbb H_{\mathcal L})_b,
        \quad
        s\in W^{k,2}(M_b;E|_{M_b}).
    \end{equation*}
    Since $h$ is constant along the fiber $M_b$, this action preserves
    $W^{k,2}$-regularity. The resulting fiberwise actions are compatible with
    local trivializations and hence define an $\mathbb H_{\mathcal L}$-module
    structure on $W^{k,2}(M/B;E)$.
\end{proof}

The following instance of an $\mathcal L$-twisted quaternionic bundle will
play a prominent role in the subsequent discussion.

\begin{lemma}
    \label{lem: twisted quaternionic spinor bundles}
    Let $(\mathfrak s,\lambda)$ be the $\mathcal L$-twisted family spin
    structure supplied by
    Proposition~\ref{Existence of an L-twisted family spin structure}, and let
    \begin{equation*}
        S=S^+\oplus S^-\longrightarrow X
    \end{equation*}
    be its spinor bundle. Then the following hold.
    \begin{enumerate}
        \item
        The bundles $S^+$ and $S^-$ are modules over
        $\pi^*\mathbb H_{\mathcal L}$.

        \item
        For every $k\geq0$, the Hilbert bundles
        $W^{k,2}(X/B;S^\pm)$ are modules over $\mathbb H_{\mathcal L}$.

        \item
        If $A$ is a unitary connection on $\mathcal L$ and
        $\widetilde A\coloneqq\pi^*A$, then the family Dirac operator
        \begin{equation*}
            D^+_{\widetilde A}\colon
            W^{k+1,2}(X/B;S^+)
            \longrightarrow
            W^{k,2}(X/B;S^-)
        \end{equation*}
        is $\mathbb H_{\mathcal L}$-linear.
    \end{enumerate}
\end{lemma}

The quaternion algebra bundle contains slightly more structure than is
required for the monopole map. The relevant symmetry is obtained by retaining
the subgroup
\begin{equation*}
    \m{Pin}(2)=S^1\cup jS^1\subset\mathbb H^\times.
\end{equation*}
Because the local quaternionic structures are related by the automorphisms
$\vartheta(z)$, the resulting symmetry is not, in general, an action of the
constant group $\m{Pin}(2)$. It is instead an action of a bundle of groups over
$B$.

The action of $\vartheta$ preserves $\m{Pin}(2)$. Consequently, it restricts
to a homomorphism
\begin{equation*}
    \vartheta\colon
    S^1\longrightarrow\m{Aut}(\m{Pin}(2)),
\end{equation*}
and we define
\begin{equation*}
    \m{Pin}(2)_{\mathcal L}
    \coloneqq
    P_{\mathcal L}\times_\vartheta\m{Pin}(2)
    \longrightarrow B.
\end{equation*}
Since the transition functions act by group automorphisms,
$\m{Pin}(2)_{\mathcal L}$ is naturally a group object in the slice category
$\cat{Top}/B$. It is important that $\m{Pin}(2)_{\mathcal L}$ is a bundle of
groups, rather than a principal $\m{Pin}(2)$-bundle. In particular, it admits
fiberwise multiplication and inversion maps
\begin{equation*}
    \mu\colon
    \m{Pin}(2)_{\mathcal L}\times_B\m{Pin}(2)_{\mathcal L}
    \rightarrow
    \m{Pin}(2)_{\mathcal L}\und{1.0cm}
    \iota\colon
    \m{Pin}(2)_{\mathcal L}
    \rightarrow
    \m{Pin}(2)_{\mathcal L},
\end{equation*}
as well as a canonical identity section
\begin{equation*}
    e\colon B\rightarrow\m{Pin}(2)_{\mathcal L},\quad
    b\mapsto[p,1],
\end{equation*}
where $p\in(P_{\mathcal L})_b$ is arbitrary. This is well-defined because
$\vartheta(z)$ fixes the identity element of $\m{Pin}(2)$ for every
$z\in S^1$.

Moreover, $\vartheta(z)$ fixes the subgroup
$S^1\subset\m{Pin}(2)$ pointwise. It follows that
$\m{Pin}(2)_{\mathcal L}$ contains a canonical normal subgroup bundle
\begin{equation*}
    \underline{S}^1_B
    \subset
    \m{Pin}(2)_{\mathcal L}.
\end{equation*}
Since $\m{Pin}(2)/S^1\cong C_2$ and the induced action of $\vartheta$ on this quotient is trivial, there is a
canonical identification
\begin{equation*}
    \m{Pin}(2)_{\mathcal L}/\underline{S}^1_B
    \cong
    \underline{C_2}_B.
\end{equation*}
We denote the resulting quotient homomorphism by
\begin{align}
    \label{epsilon map definition}
    \epsilon\colon
    \m{Pin}(2)_{\mathcal L}
    \longrightarrow
    \underline{C_2}_B.
\end{align}
Although $\m{Pin}(2)_{\mathcal L}$ admits a canonical identity section, it
does not admit a preferred global choice of the element $j$.

\begin{definition}
    Let $G\in\cat{Grp}(\cat{Top}/B)$. A \textbf{$G$-action} on an object
    $E\in\cat{Top}/B$ is a fiberwise action
    \begin{equation*}
        G\times_BE\longrightarrow E
    \end{equation*}
    satisfying the usual associativity and identity conditions. If $E$ is a
    retractive space over $B$, we additionally require its section to be fixed
    by the action.
\end{definition}

In particular, if $f\colon M\rightarrow B$ and $E\rightarrow M$ is an $\mathcal L$-twisted quaternionic bundle, we obtain actions
\begin{equation*}
    f^*\m{Pin}(2)_{\mathcal L}\times_ME
    \longrightarrow E\und{1.0cm}\Pin(2)_{\mathcal{L}}\times_BW^{k,2}(M/B,E)\rightarrow W^{k,2}(M/B,E).
\end{equation*}
We have therefore replaced the usual global $\m{Pin}(2)$-symmetry by a fiberwise action of the group bundle $\m{Pin}(2)_{\mathcal L}\rightarrow B$. This is the natural symmetry retained by the family Dirac operator and the nonlinear monopole map. We shall formulate the family monopole map and its finite-dimensional approximation as $\m{Pin}(2)_{\mathcal L}$-equivariant
maps over $B$, that is, as objects and morphisms in the categories
\begin{equation*}
    \cat{Top}^{\m{Pin}(2)_{\mathcal L}}/B
    \und{1.0cm}
    B\backslash
    \cat{Top}^{\m{Pin}(2)_{\mathcal L}}/B.
\end{equation*}

After passing to $\underline{S}^1_B$-quotients, the residual symmetry group
bundle is the constant bundle $\underline{C_2}_B=C_2\times B$. Upon subsequently forgetting the projection to $B$, or collapsing the distinguished section where appropriate, we obtain ordinary $C_2$-spaces and pointed $C_2$-spaces. The resulting objects therefore belong to
\begin{equation*}
    \cat{Top}^{C_2}
    \und{1.0cm}\cat{Top}^{C_2}_*.
\end{equation*}
In the subsequent discussion, we will encounter the trivial real line bundle
over $B$ equipped with the nontrivial action induced by $\epsilon$. We denote
this bundle by
\begin{equation*}
    \widetilde{\mathbb R}_B\longrightarrow B.
\end{equation*}
After passing to the $C_2$-quotient description, the same bundle will be
written as
\begin{equation*}
    \mathbb R^\sigma\times B\longrightarrow B,
\end{equation*}
where $\sigma$ denotes the nontrivial one-dimensional real representation of
$C_2$.
\subsection{The family monopole map}
\label{The family monopole map}

Fix $k\geq 3$.  We consider the Hilbert bundles
\begin{align*}
    \mathcal W^+
    &\coloneq W^{k+1, 2}(X/B; S^+),
    &
    \mathcal W^-
    &\coloneq W^{k, 2}(X/B; S^-),\\
    \mathcal V^+
    &\coloneqq W^{k+1, 2}(X/B; \mathrm i\Lambda^1V^*\pi),
    &
    \mathcal V^-
    &\coloneq
    W^{k, 2}(X/B; \mathrm i\Lambda^2_+V^*\pi)
    \oplus
    W^{k, 2}(X/B; \mathrm i\Lambda^0V^*\pi)/\mathrm i\mathbb R.
\end{align*}
The group bundle $\m{Pin}(2)_{\mathcal{L}}$ acts on $\mathcal W^\pm$ through their $\mathbb H_{\mathcal L}$-module structures and on $\mathcal V^\pm$ through the sign homomorphism $\epsilon$.  In a local frame, this is the familiar $\mathrm{Pin}(2)$-action: $S^1$ acts by complex multiplication on spinors and trivially on forms, while $j$ acts quaternionically on spinors and by $-1$ on forms.

\begin{definition}
    The family Seiberg--Witten monopole map is
\begin{equation*}
    \mathsf{SW}_B\colon
    \mathcal W^+\oplus\mathcal V^+
    \rightarrow
    \mathcal W^-\oplus\mathcal V^-,\quad
    (\phi,\alpha)\mapsto
    \left(
        D^+_{\widetilde A+\alpha}\phi,
        d^+\alpha-\rho^{-1}(\phi^*\phi)_0,
        d^*\alpha
    \right),
\end{equation*}
where $\rho\colon \mathrm i\Lambda^2_+\to\mathfrak{su}(S^+)$ denotes Clifford multiplication.
\end{definition}

\begin{lemma}
    The gauge symmetry and the quaternionic conjugation symmetry of the Seiberg--Witten equations imply that $\mathsf{SW}_B$ is $\mathrm{Pin}(2)_{\mathcal L}$-equivariant over $B$.  It decomposes as
\begin{equation*}
    \mathsf{SW}_B=(L_1,L_2)+Q,
\end{equation*}
where $L_1\coloneqq D^+_{\widetilde A}\colon\mathcal W^+\rightarrow\mathcal W^-$ and $L_2\coloneqq \mathrm d^+\oplus\mathrm d^*\colon\mathcal V^+\rightarrow\mathcal V^-$.
The operator $L_1$ is $\mathbb H_{\mathcal L}$-linear and Fredholm, whereas $L_2$ is real-linear, Fredholm, and $\mathrm{Pin}(2)_{\mathcal L}$-equivariant for the sign actions.  Thus $L$ is a $\mathrm{Pin}(2)_{\mathcal L}$-equivariant Fredholm morphism, but only its spinorial summand is quaternionic-linear.  The nonlinear term $Q$ is $\mathrm{Pin}(2)_{\mathcal L}$-equivariant and compact on bounded sets by Sobolev multiplication and Rellich compactness.
\end{lemma}

We now apply the standard finite-dimensional approximation of the families monopole map; see, for example, \cite[Section 2.3]{Baraglia_Konno22}.  We recall the construction in the present twisted setting in order to keep track of the full symmetry group bundle.

Since the fibres are simply connected, $L_2$ is fibrewise injective.  Its cokernel is the rank-seven bundle
\begin{equation*}
    \coker(L_2)\eqqcolon H^+\rightarrow B
\end{equation*}
of fibrewise self-dual harmonic two-forms.  By the same argument as in Subsection \ref{Cohomological reduction for n=2} $H^+$ is determined by a map 
\begin{equation*}
    B\rightarrow \m{Gr}^+_7(\mathbb{R}^{7,39}),
\end{equation*}
and hence trivial as $\m{Gr}^+_7(\mathbb{R}^{7,39})$ is contractible. Choose a sufficiently large finite-dimensional trivial subbundle
\begin{equation*}
    V^+\subset\mathcal V^+
\end{equation*}
and set
\begin{equation*}
    V^-
    \coloneqq
    L_2(V^+)\oplus H^+
    \subset\mathcal V^-.
\end{equation*}
Both bundles are $\mathrm{Pin}(2)_{\mathcal L}$-invariant, and after choosing trivializations we have
\begin{equation*}
    V^+
    \cong
    \underline{\widetilde{\mathbb R}}^{\,m}_B\und{1.0cm}
    V^-
    \cong
    \underline{\widetilde{\mathbb R}}^{\,m+7}_B.
\end{equation*}

Next choose a sufficiently large finite-rank $\mathbb H_{\mathcal L}$-submodule bundle
\begin{equation*}
    W^-\subset\mathcal W^-
\end{equation*}
which is fibrewise transverse to $\mathrm{im}(L_1)$, and define
\begin{equation*}
    W^+
    \coloneqq
    L_1^{-1}(W^-)
    \subset\mathcal W^+.
\end{equation*}
Then $W^+$ is again a finite-rank $\mathbb H_{\mathcal L}$-module bundle.  Choosing module-compatible metrics makes the orthogonal projection
\begin{equation*}
    \mathrm{pr}_{W^-\oplus V^-}\colon
    \mathcal W^-\oplus\mathcal V^-
    \rightarrow
    W^-\oplus V^-
\end{equation*}
$\mathrm{Pin}(2)_{\mathcal L}$-equivariant.

\begin{definition}
    With the above notation, we obtain the finite-dimensional approximation
\begin{equation*}
    \mathsf{SW}^{\mathrm{app}}_B\colon
    W^+\oplus V^+
    \rightarrow
    W^-\oplus V^-,\quad
    (\phi,\alpha)\mapsto
    L(\phi,\alpha)
    +\mathrm{pr}_{W^-\oplus V^-}Q(\phi,\alpha).
\end{equation*}
\end{definition}

The construction uses only orthogonal projections, the Fredholm property of $L$, and uniform compactness of $Q$ over the compact base. The chosen orthogonal projections are $\m{Pin}(2)_{\mathcal{L}}$-equivariant. Since both $L$ and $Q$ are $\m{Pin}(2)_{\mathcal{L}}$-equivariant, so is $\mathsf{SW}^{\mathrm{app}}_B$.

\begin{rem}
    Different sufficiently large choices of $V^+$ and $W^-$ agree up to stabilization and pointed $\mathrm{Pin}(2)_{\mathcal L}$-equivariant homotopies over $B$. Consequently, the approximations determine a well-defined stable family Bauer--Furuta class in the appropriate parametrized equivariant stable homotopy theory; schematically
    \begin{equation*}
        \mathsf{BF}(X,\mathcal{L})= \underset{V^+\oplus W^+\subset \mathcal{V}^+\oplus\mathcal{W}^+}{\mathrm{colim}}\mathsf{SW}^{\m{app}}_B.
    \end{equation*}
    We will not make this construction explicit, since below we only use one sufficiently large representative.
\end{rem}

For a sufficiently large approximation, compactness of the Seiberg--Witten equations gives constants $R\gg 1$ and $\epsilon>0$ such that, uniformly over $B$,
\begin{equation*}
    ||(\phi,\alpha)||_{W^{k+1,2}}=R
    \qquad\Longrightarrow\qquad
    ||\mathsf{SW}^{\mathrm{app}}_B(\phi,\alpha)||_{W^{k,2}}\geq\epsilon.
\end{equation*}
Radial collapse outside the $\epsilon$-ball therefore produces a pointed $\mathrm{Pin}(2)_{\mathcal L}$-equivariant map over $B$,
\begin{equation*}
    \mathsf{SW}^{\mathrm{app}}_{B,+}
    \colon
    S_B^{W^+\oplus V^+}
    \rightarrow
    S_B^{W^-\oplus V^-}.
\end{equation*}
Equivalently, after collapsing the sections at infinity, it is a map of Thom spaces
\begin{equation*}
    \mathsf{SW}^{\mathrm{app}}_{B,+}\colon
    S^{m\widetilde{\mathbb R}}\wedge \m{Th}(W^+)
    \rightarrow
    S^{(m+7)\widetilde{\mathbb R}}\wedge \m{Th}(W^-).
\end{equation*}

The $S^1$-fixed vectors in $W^\pm$ are zero, whereas $S^1$ acts trivially on $V^\pm$.  Thus the $S^1$-fixed locus of $\mathsf{SW}^{\mathrm{app}}_{B,+}$ is
\begin{equation*}
    (\mathsf{SW}^{\mathrm{app}}_{B,+})^{S^1}\colon
    S^{m\widetilde{\mathbb R}_B}\rightarrow
    S^{(m+7)\widetilde{\mathbb R}_B}.
\end{equation*}
On the reducible locus $\{0\}\oplus V^+$, the nonlinear term vanishes and the monopole map is $L_2$.  After an equivariant homotopy supported away from the section at infinity, we may therefore arrange that $(\mathsf{SW}^{\mathrm{app}}_{B,+})^{S^1}$ is exactly the constant family of standard inclusions
\begin{equation*}
    S^{m\widetilde{\mathbb R}_B}
    \lhook\joinrel\rightarrow
    S^{(m+7)\widetilde{\mathbb R}_B}.
\end{equation*}

\begin{lemma}
    \label{lem: locally quaternionic injection}
    There exists an injective $\mathbb H_{\mathcal L}$-linear bundle map
    \begin{equation*}
        W^-\lhook\joinrel\rightarrow W^+.
    \end{equation*}
    Consequently, we can choose an orthogonal decomposition of $\mathbb H_{\mathcal L}$-module bundles
\begin{equation*}
    W^+
    \cong
    W^-\oplus W^\perp,
\end{equation*}
where $W^\perp$ has quaternionic rank two.
\end{lemma}

\begin{proof}
    Since the spin Dirac operator on $F\cong_{\cat{Top}}E(4)$ has complex index $-\sigma(E(4))/8=4$, its quaternionic index is two. Hence
    \begin{equation*}
        \mathrm{rank}_{\mathbb H}W^+
        -\mathrm{rank}_{\mathbb H}W^-
        =2.
    \end{equation*}
    If $r=\mathrm{rank}_{\mathbb H}W^-$, the space of quaternionic-linear injections
    \begin{equation*}
        \mathbb H^r\lhook\joinrel\rightarrow\mathbb H^{r+2}
    \end{equation*}
    deformation retracts onto the quaternionic Stiefel manifold
    \begin{equation*}
        V_m(\mathbb H^{r+2})
        \cong
        \mathrm{Sp}(r+2)/\mathrm{Sp}(2).
    \end{equation*}
    This Stiefel manifold is $10$-connected.  The bundle over $B$ whose fiber consists of $\mathbb H_{\mathcal L}$-linear injections $W^-_b\to W^+_b$ therefore has $10$-connected fibers.  Since $\dim(B)=4$, obstruction theory supplies a global section, which is the desired map.
\end{proof}

\subsection{Desuspension and reduction to a $C_2$-equivariant problem} \label{Desuspension and reduction to a $C_2$-equivariant problem}

Using Lemma \ref{lem: locally quaternionic injection}, let us fix an $\mathbb{H}_{\mathcal{L}}$-linear decomposition $W^+ \cong W^- \oplus W^\perp$ with $W^\perp$ of quaternionic rank two. Hence the finite-dimensional approximation of the previous section may be regarded as a pointed $\mathrm{Pin}(2)_{\mathcal{L}}$-equivariant map over $B$
\begin{equation*}
    \mathsf{SW}^{\mathrm{app}}_{B,+}\colon
    S^{m\widetilde{\mathbb R}_B}\wedge_B S^{W^-} \wedge_B S^{W^\perp} 
    \rightarrow
    S^{(m+7)\widetilde{\mathbb R}_B}\wedge_B S^{W^-}.
\end{equation*}
We will desuspend the common factor $S^{(m-1)\widetilde{\mathbb R}_B}\wedge_B S^{W^-}$.

\begin{prop}\label{prop: desuspension}
    There exists a pointed $\mathrm{Pin}(2)_{\mathcal{L}}$-equivariant map
    \begin{equation*}
        \mathsf{sw}_B\colon S^{\widetilde{\mathbb R}_B}\wedge_B S^{W^\perp} 
        \rightarrow
        S^{8\widetilde{\mathbb R}_B}
    \end{equation*}
    satisfying the following properties:
    \begin{enumerate}
        \item The restriction $\mathsf{sw}_B$ to the $S^1$-fixed locus is the standard fiberwise inclusion
        \begin{equation*}
            S^{\widetilde{\mathbb R}_B} \hookrightarrow S^{8\widetilde{\mathbb R}_B};
        \end{equation*}
        \item After suspension by $S^{(m-1)\widetilde{\mathbb R}_B}\wedge_B S^{W^-}$, the map $\mathsf{sw}_B$ is pointed $\mathrm{Pin}(2)_{\mathcal{L}}$-equivariantly homotopic over $B$ to $\mathsf{SW}^{\mathrm{app}}_{B,+}$ relative to the $S^1$-fixed locus.
    \end{enumerate}
\end{prop}

For the proof we need the following relative version of the Freudenthal suspension theorem.

\begin{lemma}\label{relative Freudenthal}
    Let $A$ be a pointed $\mathrm{Pin}(2)$-space which is (nonequivariantly) $p$-connected. Let $V$ be a finite-dimensional real $\mathrm{Pin}(2)$-representation. Suppose that $0 \leq q \leq 2p$ and that the map 
    \begin{equation*}
        f\colon S^{V}\wedge  D^{q}_{+} \wedge  \mathrm{Pin}(2)_{+}\to S^{V}\wedge A
    \end{equation*}
    is $\mathrm{Pin}(2)$-equivariant and satisfies 
    \begin{equation*}
        f|_{S^{V}\wedge  \partial D^{q}_{+} \wedge  \mathrm{Pin}(2)_{+}} = \mathrm{id}_{S^{V}}\wedge g_+
    \end{equation*}
    for some equivariant map $g\colon \partial D^{q}\times  \m{Pin}(2)\to A$.\footnote{Here $g_{+}$ means adding a base point to the domain of $g$ and sending it to the base point in $A$.} 

    Then, there exists a  $\mathrm{Pin}(2)$-equivariant map $f'\colon D^q \times \mathrm{Pin}(2) \to A$ such that
    \begin{equation*}
        f'|_{\partial D^q \times \mathrm{Pin}(2)} = g
    \end{equation*}
    and $f$ and $\mathrm{id}_{S^V} \wedge f'_+$ are $\mathrm{Pin}(2)$-equivariantly homotopic relative to $S^{V}\wedge  \partial D^{q}_{+} \wedge  \mathrm{Pin}(2)_{+}$.
\end{lemma}
\begin{proof}
    For any pointed $\mathrm{Pin}(2)$-spaces $Y$ and $Z$, restriction to the identity gives the usual identification
    \begin{equation}\label{equivariant non equivariant correspondence}
        \mathrm{Hom}_{\cat{Top}_*}^{\mathrm{Pin}(2)}(Y \wedge \mathrm{Pin(2)}_+, Z)  \cong \mathrm{Hom}_{\cat{Top}_*}(Y, Z)
    \end{equation}
    Under this identification, $f$ and $g$ correspond to nonequivariant maps
    \begin{equation*}
        f_{e}\colon S^{V}\wedge D^{q}_{+}\wedge \{e\}_{+}\to S^{V}\wedge A
    \end{equation*}
    and 
    \begin{equation*}
        g_{e}\colon \partial D^{q}_{+}\wedge \{e\}_{+}\to A.
    \end{equation*}
    The boundary condition $f|_{S^{V}\wedge  D^{q}_{+} \wedge  \mathrm{Pin}(2)_{+}} = \mathrm{id}_{S^{V}}\wedge g_+$ means that 
    \begin{equation*}
        f_e |_{ S^V \wedge \partial D^q_+} = \mathrm{id}_{S^V} \wedge (g_e)_+.
    \end{equation*}
    Since $\Omega$ and $\Sigma$ are mutually adjoint functors, we get
    \begin{equation*}
        \hat{f}_e\colon D^q \to \Omega^{\dim V} \Sigma^{\dim V} A.
    \end{equation*}
    If we write 
    \begin{equation*}
        \eta_A\colon A \to \Omega^{\dim V} \Sigma^{\dim V} A
    \end{equation*}
    for the unit of the suspension-loop adjunction (the adjoint to the identity map $\mathrm{id}_{\Sigma^{\dim V}}$), then our assumptions yield
    \begin{equation*}
        \hat{f}_e|_{\partial D^q} = \eta_A \circ g_e.
    \end{equation*}
    Now, since $A$ is $p$-connected and $q \leq 2p$, Freudenthal suspension theorem gives
    \begin{equation*}
        \pi_q( \Omega^{\dim V} \Sigma^{\dim V} A, A) = 0.
    \end{equation*}
    It follows that $\hat{f}_e$ is homotopic relative to $\partial D^q$ to a map
    \begin{equation*}
        f'_e \colon D^q \to A \quad \text{with } f'_e|_{\partial D^q} = g_e.
    \end{equation*}
    Passing back to a $\mathrm{Pin}(2)$-equivariant map via \eqref{equivariant non equivariant correspondence} yields the claim.
\end{proof}

\begin{proof}[Proof of Proposition \ref{prop: desuspension}]
    We pick a triangulation of $B$ and consider the induced filtration 
    \begin{equation*}
        \emptyset=B_{0}\subset B_{1}\subset B_{2}\subset \cdots \subset B_{m}=B,
    \end{equation*}
    where $B_{i}=B_{i-1}\cup_{\partial \Delta_{i}} \Delta_{i}$ is obtained by adjoining a simplex $\Delta_{i}$. 
    
    We will inductively construct $\mathrm{Pin}(2)_{\mathcal{L}|_{B_i}}$-equivariant maps    
    \begin{equation*}
        \mathsf{sw}_{i}\colon S^{\widetilde{\mathbb{R}}_{B_i}} \wedge_{B_i} S^{W^\perp|_{B_i}} \to S^{8\widetilde{\mathbb{R}}_{B_i}} 
    \end{equation*}
    such that
    \begin{itemize}
        \item $\mathsf{sw}_{i}$ extends $\mathsf{sw}_{i-1}$,
        \item The restriction of $\mathsf{sw}_{i}$ to the $S^{1}$-fixed points equals the inclusion $S^{\widetilde{\mathbb{R}}_{B_i}}\hookrightarrow S^{8\widetilde{\mathbb{R}}_{B_i}}$ and 
        \item The suspension of $\mathsf{sw}_{i}$ via $S^{(m-1)\widetilde{\mathbb{R}}_{B_i}} \wedge_{B_i} S^{W^-|_{B_i}}$ is $\mathrm{Pin}(2)_{\mathcal{L}|_{B_i}}$-equivariantly homotopic to $\mathsf{SW}^{\mathrm{app}}_{B_i, +}$ relative to the $S^1$-fixed locus.
    \end{itemize}

    Suppose that $\mathsf{sw}_{i-1}$ has been constructed. By the homotopy extension property, we may arrange that over $B_{i-1}$ the suspension of $\mathsf{sw}_{i-1}$ agrees with $\mathsf{SW}^{\mathrm{app}}_{+}$. Since $\Delta_i$ is contractible, we may choose trivializations
    \begin{equation*}
        \mathrm{Pin}(2)_{\mathcal{L}|_{\Delta_i}} \cong \Delta_i \times \mathrm{Pin}(2), \quad W^- |_{\Delta_i} \cong \Delta_i \times \mathbb{H}^r \quad \text{and } W^\perp |_{\Delta_i} \cong \Delta_i \times \mathbb{H}^2. 
    \end{equation*}
     After collapsing the sections, pointed maps over $\Delta_i$ are identified with ordinary pointed maps.\footnote{More precisely, if $\Delta$ is contractible and $E \in \Delta\backslash\mathbf{Top}/\Delta$ with section $s\colon \Delta \to E$, then pointed maps over $\Delta$ of the form $E \to \Delta \times A$ are naturally identified with pointed maps $E/s(\Delta) \to A$. The same holds for pointed homotopies.} Thus consider
     \begin{equation*}
         Y_1 = S^{\widetilde{\mathbb{R}}} \wedge \Delta_{i, +} \wedge S^{2\mathbb{H}} \quad \text{and} \quad Y_2= (S^{\widetilde{\mathbb{R}}} \wedge (\partial\Delta_i)_+ \wedge S^{2\mathbb{H}}) \cup (S^{\widetilde{\mathbb{R}}}  \wedge  \Delta_{i, +}).
     \end{equation*}
     Here the second term in $Y_2$ corresponds to the $S^1$-fixed locus, since $(S^{2 \mathbb{H}})^{S^1} = S^0$. The already constructed map $\mathsf{sw}_{i-1}$ gives a map on the first part of $Y_2$ while on the second part, we have the standard inclusion $S^{\widetilde{\mathbb{R}}} \hookrightarrow S^{8\widetilde{\mathbb{R}}}$. By the inductive hypothesis, these agree on their intersection and therefore determine a pointed $\mathrm{Pin}(2)$-equivariant map 
     \begin{equation*}
         g\colon Y_2 \to S^{8\widetilde{\mathbb{R}}}.
     \end{equation*}
     Using the homotopy extension property and the inductive hypothesis once again, we may arrange that, under the above identifications,
     \begin{equation*}
         \mathsf{SW}^{\mathrm{app}}_{B, +}\Big|_{S^{(m-1)\widetilde{\mathbb{R}}}\wedge S^{r \mathbb{H}}\wedge Y_2} = \mathrm{id}_{S^{(m-1)\widetilde{\mathbb{R}}} \wedge S^{r \mathbb{H}}} \wedge g.
     \end{equation*}
     Away from the $S^1$-fixed locus, the action of $\mathrm{Pin}(2)$ on $\mathbb{H}^2 \setminus \{ 0 \}$ is free. Hence the relative $\mathrm{Pin}(2)$-CW complex $(Y_1, Y_2)$ is obtained by attaching free cells $D^q \times \mathrm{Pin}(2)$. Since $\dim Y_1 = \dim \Delta_i +9$, these may be taken with $q \leq \dim \Delta_i +8 \leq 12$.
     
     We now attach these cells one at a time. On each cell the restriction of $\mathsf{SW}^{\mathrm{app}}_{B, +}$ has the form 
     \begin{equation*}
         f\colon S^V \wedge D^q_+ \wedge \mathrm{Pin}(2)_+ \to S^V \wedge S^{8 \widetilde{\mathbb{R}}},
     \end{equation*}
     where 
     \begin{equation*}
         V = (m-1) \widetilde{\mathbb{R}} \oplus r \mathbb{H}
     \end{equation*}
     and on the boundary 
     \begin{equation*}
         f|_{S^{V}\wedge  D^{q}_{+} \wedge  \mathrm{Pin}(2)_{+}} = \mathrm{id}_{S^{V}}\wedge h_+,
     \end{equation*}
     where the meaning of $h$ is as follows: if $Y \subset Y_1$ is the union of $Y_2$ with the cells already attached and $\varphi\colon \partial D^q \times \mathrm{Pin}(2) \to Y$ is the attaching map, then 
     \begin{equation*}
         h= \mathsf{sw}_Y \circ \varphi,
     \end{equation*}
     where $\mathsf{sw}_Y $ is the desuspended map constructed so far.
     
     Since $S^{8 \widetilde{\mathbb{R}}}$ is $7$-connected and $q \leq 12 \leq 14$, Lemma \ref{relative Freudenthal} provides an equivariant extension
     \begin{equation*}
         h'\colon D^q \times \mathrm{Pin}(2) \to S^{8\tilde{\mathbb{R}}}
     \end{equation*}
     with a $\mathrm{Pin}(2)$-equivariant homotopy 
     \begin{equation*}
         f \simeq \mathrm{id}_{S^V} \wedge h'_+ \quad \text{rel } S^V\wedge \partial D^q_+ \wedge \mathrm{Pin}(2)_+.
     \end{equation*}
     Thus $h'$ extends the desuspended map over the new cell while the relative homotopy extends the comparison homotopy with $\mathsf{SW}^{\textrm{app}}_{B, +}.$ Repeating this procedure over all the free cells of $(Y_1, Y_2)$, we obtain a pointed $\mathrm{Pin}(2)$-equivariant map
     \begin{equation*}
         \overline{\mathsf{sw}}_{\Delta_i}: Y_1 \to S^{8\widetilde{\mathbb{R}}}
     \end{equation*}
     Passing back to pointed spaces over $\Delta_i$, the map $\overline{\mathsf{sw}}_{\Delta_i}$ determines 
     \begin{equation*}
         \mathsf{sw}_{\Delta_i} \colon S^{\widetilde{\mathbb{R}}_{\Delta_i}} \wedge_{\Delta_i} S^{W^\perp |_{\Delta_i}} \to S^{8 \widetilde{\mathbb{R}}_{\Delta_i}}.
     \end{equation*}
     Since $\mathsf{sw}_{\Delta_i}$ agrees with $\mathsf{sw}_{i-1}$ over $\partial\Delta_i$, the maps glue. Continuing inductively over all simplices yields the required map $\mathsf{sw}_B$.
\end{proof}

We now derive from the existence of $\mathsf{sw}_B$ a $C_2$-equivariant consequence which will be contradicted in the next sections, ultimately leading to the conclusion of proof of Theorem \ref{mainthm: X non spin}.

We have a cofiber sequence of retractive $\m{Pin}(2)_{\mathcal L}$-spaces
\begin{equation*}
    S(W^\perp)_{+B}
    \rightarrow
    S^{0_B}
    \rightarrow
    S^{W^\perp}.
\end{equation*}
After smashing with $S^{\widetilde{\mathbb R}_B}$ and mapping into $S^{8\widetilde{\mathbb R}_B}$, the defining universal property of the cofiber shows that the existence of $\mathsf{sw}_B$ forces the composite
\begin{equation*}
    S^{\widetilde{\mathbb R}_B}
    \wedge_B S(W^\perp)_{+B}
    \rightarrow
    S^{\widetilde{\mathbb R}_B}
    \lhook\joinrel\longrightarrow
    S^{8\widetilde{\mathbb R}_B}
\end{equation*}
to be $\m{Pin}(2)_{\mathcal L}$-equivariantly null-homotopic over $B$; we refer the reader to \cite[Section 12]{CrabbJames1998} for a comprehensive treatment of fibrewise pointed cofibrations.

The canonical subgroup $\underline{S}^1_B\subset\m{Pin}(2)_{\mathcal L}$ acts freely on $S(W^\perp)$, and
\begin{equation*}
    P\coloneqq
    S(W^\perp)/\underline{S}^1_B
    \cong
    \mathbb P_{\mathbb C}(W^\perp)
\end{equation*}
is a $\mathbb{CP}^3$-bundle over $B$. Since $\m{Pin}(2)_{\mathcal L}/\underline{S}^1_B\cong\underline{C_2}_B$, the quotient $P$ carries a genuine global $C_2$-action.  Explicitly, if $c\in C_2$, $[v]\in P_b$, and $g\in\m{Pin}(2)_{\mathcal L}|_b$ is any lift of $c$, then
\begin{equation*}
    c\cdot[v]\coloneqq[g\cdot v].
\end{equation*}
This is independent of the choice of the lift because two lifts differ by an element of $S^1$, which does not change the complex line $[v]$.

Passing to the quotient of the preceding null-homotopy by $S^1$, and then collapsing the base against the trivial target sphere bundle, the representation $\widetilde{\mathbb R}$ becomes the sign representation $\sigma$ of $C_2$.  We obtain the following consequence.

\begin{lemma}
    \label{lem: sw map implies null homotopic}
    If the map $\mathsf{sw}_B$ exists, then the $C_2$-equivariant map
    \begin{align}
    \label{eq: P to S}
        f\colon
        S^\sigma\wedge P_+
        \xrightarrow{\m{id}\wedge q_P}
        S^\sigma\wedge S^0
        \cong S^\sigma
        \hookrightarrow
        S^{8\sigma}
    \end{align}
    is $C_2$-equivariantly null-homotopic.  Here $q_P\colon P_+\to S^0$ collapses $P$ to the non-basepoint.
\end{lemma}

The remainder of the proof will contradict Lemma \ref{lem: sw map implies null homotopic} by showing that $f$ induces a nonzero map in $C_2$-equivariant $K$-theory.

\subsection{Equivariant $K$-theory of $\mathbb{P}_{\mathbb{C}}(W^{\perp})$}
\label{Equivariant K-theory of PcW}

Note that while the $\Pin(2)_{\mathcal{L}}$-action on $W^{\perp}$ only exists in in $\cat{Top}/B$ , the induced $\underline{C_2}_B$-action on $P$ descends to a $C_2$-action in $\cat{Top}$. To prove Theorem \ref{mainthm: X non spin}, we will show that that composition \eqref{eq: P to S} is not $C_2$-equivariantly null-homotopic, by proving that the induced map on equivariant $K$-theory is nontrivial.

We introduce some notation. Let $A$ be a pointed space with a $C_2$-action. For $p,q\in \mathbb{Z}$ We define 
\begin{equation*}
\widetilde{K}^{p+q\sigma}_{C_2}(A)\coloneqq \begin{cases}
\widetilde{K}_{C_2}(A) &\text{ if $p$ even, $q$ even}\\
\widetilde{K}_{C_2}(S^{\sigma}\wedge A) &\text{ if $p$ even, $q$ odd}\\
\widetilde{K}_{C_2}(S^1\wedge A) &\text{ if $p$ odd, $q$ even}\\
\widetilde{K}_{C_2}(S^1\wedge S^{\sigma}\wedge A) &\text{ if $p$ odd, $q$ odd}
\end{cases}.
\end{equation*}
We set $K^{p+q\sigma}_{C_2}(A)\coloneqq \widetilde{K}^{p+q\sigma}_{C_2}(A_{+})$. For a $C_2$-subspace $A'$ of $A$ and a class $c\in K^{p+q\sigma}_{C_2}(A)$, we use $c|_{A'}\in K^{p+q\sigma}_{C_2}(A')$ to denote the pull-back of $c$. For any pointed $C_2$-space $A$, we consider the cofiber sequence 
\begin{equation*}
\cdots \to S(\sigma)_{+}\wedge A\to S^0\wedge A\to S^{\sigma}\wedge A\to  S^1\wedge S(\sigma)_{+}\wedge A\to S^1\wedge A\to S^1\wedge S^{\sigma}\wedge A\to\cdots
\end{equation*}
Note that from the natural equivalence $\widetilde{K}_{C_2}(S(\sigma)_{+}\wedge -)\cong \widetilde{K}(-)$, we obtain the exact sequence 
\begin{equation}
\label{eq: exact 6-gon v1}
\begin{tikzcd}
	& {\widetilde{K}^\sigma_{C_2}(A)} & {\widetilde{K}_{C_2}(A)} & \\
	{\widetilde{K}^1_{C_2}(A)} &&& {\widetilde{K}(A)} \\
	& {\widetilde{K}^1_{C_2}(A)} & {\widetilde{K}^{1+\sigma}_{C_2}(A)}
	\arrow[from=1-2, to=1-3]
	\arrow[from=1-3, to=2-4]
	\arrow["{\mathsf{tran}}"{description}, from=2-1, to=1-2]
	\arrow[from=2-4, to=3-3]
	\arrow[from=3-2, to=2-1]
	\arrow[from=3-3, to=3-2]
\end{tikzcd}
\end{equation}

We call the map $\mathsf{tran}\colon\widetilde{K}^{1}(A)\to \widetilde{K}^{\sigma}_{C_2}(A)$ the \textbf{transfer map}.

\subsubsection{The projective bundle tower}
A key step in our argument is to understand the group $K^{\sigma}_{C_2}(P)$. For this, we further explore the fibration structure of $P$. Since $W^{\perp}$ is an $\mathbb{H}_{\mathcal{L}}$-module, we have a $\m{Sp}(1)_{\mathcal{L}}$-action on $S(W^\perp)$, where $\m{Sp}(1)_{\mathcal{L}} \coloneq P_{\mathcal{L}} \times_{\vartheta} \mathrm{Sp}(1) \subset \mathbb{H}_{\mathcal{L}}$. We denote the orbit space by 
\begin{equation*}
    \pi_2: P_{\mathbb{H}} \coloneqq \mathbb{P}_{\mathbb{H}}(W^{\perp}) \cong S(W^\perp)/ \m{Sp}(1)_{\mathcal{L}}  \rightarrow B.
\end{equation*}
Moreover, let
\begin{equation*}
    \pi_1: P \to P_{\mathbb{H}}
\end{equation*}
to be the natural quotient map. We have the following lemma.
\begin{lemma}
    \label{lem: P(W) bundle structure}
    The following holds:
    \begin{enumerate}
        \item We have a fiber bundle  $S^2\hookrightarrow P\to P_{\mathbb{H}}$ whose structure group is $\m{SO}(3)$.
        \item Under this fibration structure on $P$, the $C_2$-action is exactly the fiberwise antipodal map.
    \end{enumerate}
\end{lemma}

\begin{proof} 
The fibers of $\pi_1$ are the complex lines in the quaternionic space, hence $\m{Sp}(1)/S^1\cong S^2$. The transition functions act on this $S^2$ through $\m{Aut}_{\mathbb{R}\m{alg}}(\mathbb{H}) \cong \m{SO}(3)$, so $\pi_1: P \to P_{\mathbb{H}}$ is an $S^2$-bundle with structure group $\m{SO}(3)$. Under this identification, the nontrivial element of $C_2$ acts by the antipodal map on $S^2$.
\end{proof}
Because of Lemma \ref{lem: P(W) bundle structure}, we will denote the fiber $S^2$ by $S(3\sigma)$, the unit sphere in $3\sigma$, to emphasize that the $C_2$-action is the antipodal map. 

Note the bundle structure $\mathbb{CP}^{3}\hookrightarrow P\to B$. The bundle map $\pi_0\colon P\to B$ descends to a map $\pi_{2}\colon P_{\mathbb{H}}\to B$, which is a fiber bundle with fiber $S^4$. Since $\dim B =4$ and $\pi_i(S^4) = 0$, for $i \leq 3$, obstruction theory gives a section $s: B \to P_{\mathbb{H}}$. The section $s$ determines an $\mathbb{H}_{\mathcal{L}}$-line subbundle $W^{\perp}_0\subset W^{\perp}$. 

We have decomposed the bundle map $\pi_0\colon P\to B$ into the following fibration tower
\begin{equation}
\label{eq: fibration tower}
\begin{tikzcd}
	{S(3\sigma)} && P \\
	\\
	{S^4} && {P_{\mathbb{H}}} \\
	\\
	&& B
	\arrow["{i_1}"{description}, hook, from=1-1, to=1-3]
	\arrow["{\pi_1}"{description}, from=1-3, to=3-3]
	\arrow["{\pi_0}"{description}, curve={height=-18pt}, from=1-3, to=5-3]
	\arrow["{i_2}"{description}, hook, from=3-1, to=3-3]
	\arrow["{\pi_2}"{description}, from=3-3, to=5-3]
	\arrow["s"{description}, curve={height=-18pt}, from=5-3, to=3-3]
\end{tikzcd}
\end{equation}

\subsubsection{The equivariant Atiyah--Hirzebruch spectral sequence}

Now we consider the Atiyah--Hirzebruch spectral sequence associated to the bundle $\pi_1$. Since the structure group $\m{SO}(3)$ is connected, its action on the $K$-groups of the fiber is trivial, so the associated local coefficient systems are constant. We write
\begin{equation}\label{eq: AHSS}
    E^{p,q}_{2}(\pi_1)=H^{p}(P_{\mathbb{H}};K_{C_2}^{q+\sigma}( S(3\sigma)))\implies K^{p+q+\sigma}_{C_2}(P).
\end{equation}

\begin{lemma}\label{lem: K(3sigma)}
There exist isomorphisms 
\begin{equation*}
    K_{C_2}^{q+\sigma}( S(3\sigma))\cong \widetilde{K}^{q}(\mathbb{RP}^3)\cong \begin{cases} \mathbb{Z}_2 &\text{ when $q$ is even};\\
\mathbb{Z}&\text{ when $q$ is odd}. 
\end{cases}
\end{equation*}
\end{lemma}
\begin{proof}
The first isomorphism follows from the fact that 
$(S^{\sigma}\wedge S(3\sigma)_{+})/C_2\cong \mathbb{RP}^3$; the second isomorphism follows from the fact that $\mathbb{RP}^3$ is stably homotopy equivalent to $S^3\vee \mathbb{RP}^2$.
\end{proof}

Since $H^\bullet(B;\mathbb{Z})$ is supported even gradings, the Serre spectral sequence for $\pi_{2}$ collapses on the $E_2$-page. Hence there is a unique class $b\in H^{4}(P_{\mathbb{H}};\mathbb{Z})$ such that $i^*_{2}(b)$ is the generator of $H^{4}(S^4;\mathbb{Z})$ and $s^*(b)=0$. For $R=\mathbb{Z}$ or $\mathbb{Z}_2$, by the Leray-Hirsch theorem, we have the isomorphism
\begin{equation*}
H^\bullet(P_{\mathbb{H}};R)\cong H^\bullet(B;R)\langle 1, b\rangle
\end{equation*}
as modules over $H^\bullet(B;R)$. By Lemma \ref{lem: K(3sigma)}, we conclude the following corollary.

\begin{cor}
The second spectral page of the Atiyah--Hirzebruch spectral sequence associated to $\pi_1$ is isomorphic
\begin{equation*}
   E^{p,q}_{2}(\pi_1)\cong\begin{cases}
\mathbb{Z}_2\quad &p\in \{0,8\}, q\in 2\mathbb{Z}\\
\mathbb{Z}^{\oplus k+l}_2\quad &p\in \{2,6\}, q\in 2\mathbb{Z}\\
\mathbb{Z}_2\oplus \mathbb{Z}_2\quad &p=4, q\in 2\mathbb{Z}\\
\mathbb{Z}\quad &p\in \{0,8\}, q\in 2\mathbb{Z}+1\\
\mathbb{Z}^{\oplus k+l}\quad &p\in \{2,6\}, q\in 2\mathbb{Z}+1\\
\mathbb{Z}\oplus \mathbb{Z}\quad &p=4, q\in 2\mathbb{Z}+1\\
\end{cases} 
\end{equation*}
\end{cor}

\begin{lemma}\label{lem: AHSS collapses}
The spectral sequence \eqref{eq: AHSS} collapses on the second page.     
\end{lemma}

\begin{proof}
Let $E^{p, q}_r(\pi_1)$ denote the equivariant spectral sequence \eqref{eq: AHSS}. Since $H^*(P_{\mathbb{H}}; \mathbb{Z})$ is concentrated in even degrees, a potentially nonzero differential $d_r$ must have $r$ even. If $q$ is even, then $E^{p,q}_r(\pi_1)$ is torsion while $E^{p+r,q+1-r}_r(\pi_1)$ is free, so such a differential vanishes. It therefore remains to consider differentials with $q$ odd.

Consider the ordinary (that is: not equivariant) Atiyah--Hirzebruch spectral sequence
\begin{equation*}
    \bar{E}^{p, q}_2(\pi_1) = H^p(P_{\mathbb{H}}; K^{q+1}(S^2)) \implies K^{p+q+1}(P).
\end{equation*}
This spectral sequence collapses on the second page by parity. The transfer map from \eqref{eq: exact 6-gon v1}
\begin{equation*}
    \mathsf{tran}\colon K^{\bullet+1}(P)\to K^{\bullet+\sigma}_{C_2}(P)  
\end{equation*}
induces a natural morphism
\begin{equation*}
    \bar{E}^{p, q}_r(\pi_1) \longrightarrow E^{p, q}_r(\pi_1).
\end{equation*}
For odd $q$, the corresponding map on the fiber coefficients is surjective. Indeed, on the reduced summand it is identified with 
\begin{equation*}
    \widetilde{K}^{q}(\mathbb{RP}^{3}/\mathbb{RP}^{2})\to \widetilde{K}^{q}(\mathbb{RP}^{3}),
\end{equation*}
induced by the quotient map $\mathbb{RP}^{3} \to \mathbb{RP}^{3}/\mathbb{RP}^{2}$ and this map is an isomorphism for odd $q$.

Suppose that $d_r$ is the first nonzero differential on the $r$-th page, and let $x \in E^{p, q}_r(\pi_1)$ with $q$ odd. By the preceding surjectivity, $x$ lifts to some $\bar{x} \in \bar{E}^{p, q}_r(\pi_1)$. Naturality and the collapse of $\bar{E}_r(\pi_1)$ give
\begin{equation*}
    d_r (x) = \mathsf{tran}(d_r(\bar{x})) = 0,
\end{equation*}
a contradiction. Hence all differentials vanish and \eqref{eq: AHSS} collapses on the second page.
\end{proof}

\subsubsection{Torsion estimates for the class $\gamma$}

Consider the subspaces $B_0\subset B_{2}\subset B$ where $B_0$ is a point and $B_2=B\setminus \mathring{D}^4\simeq \vee_{k+l} S^2$. We define various subspaces $P=P(W^{\perp})$ as follows 
\begin{equation*}
\begin{split}
Q\coloneqq &\pi_0^{-1}(B_2)\\
Q'\coloneqq &\pi_0^{-1}(B_0)\cong \mathbb{CP}^3\\
T\coloneqq &\pi^{-1}_{1}(s(B))=P(W^{\perp}_0)\\
R\coloneqq &T\cap Q=\pi^{-1}_{1}(s(B_2))=P(W^{\perp}_0|_{B_2})\\
R'\coloneqq &\pi^{-1}_{1}(s(B_0))\cong S(3\sigma)
\end{split}.
\end{equation*}
Then we define K-theoretic classes $\alpha_{1}, \dots, \alpha_{k+l}, \gamma$ as follows: 
\begin{itemize}
    \item Let $\gamma$ be  any element in $K^{\sigma}_{C_2}(P)$ such that $\gamma|_{R'}\neq 0$. We set  $\gamma_{A}=\gamma|_{A}$ for $A\subset P$. 
    \item For $1\leq i\leq k+l$, consider the line bundles $\mathbb{C}\hookrightarrow \mathcal{L}_{i}\to B$ such that $c_{1}(\mathcal{L}_i)=a_{i}$. We let $\alpha_{i}=[\mathcal{L}_{i}]-[\mathbb{C}]\in K(B)$. Then we have the ring isomorphism 
\begin{equation*}
    K(B)=\mathbb{Z}\langle\alpha_1,\cdots,\alpha_{k+l}\rangle/\sim 
\end{equation*}
where the relation  $\sim$ is generated by 
\begin{equation*}
    \alpha^2_{1}=\cdots=\alpha^2_{k}=-\alpha^2_{k+1}=\cdots =-\alpha^2_{k+l} \text{ and }\alpha_{i}\alpha_{j} =0\ \text{ for all } i\neq j.  
\end{equation*}
Recall that $w_{2}(V\pi)\equiv \pi^*(a)\pmod 2$ with $a=\sum^{k+l}_{i=1}s_{i}a_{i}$ for $s_{i}\in \{0,1\}$. 
We let $\alpha=\sum ^{k+l}_{i=1}s_{i}\alpha_{i}$. 
\end{itemize}

Note that for any $C_2$-subspace $A$ of $P$, the composition $A\hookrightarrow P\xrightarrow{\pi_0}B$ endows $K^{p+q\sigma}_{C_2}(A)$ with a $K(B)$-module structure. 

We will compare the restrictions of $\gamma$ to the two subspaces $T$ and $Q$. The key point is that these restrictions have different torsion behavior: we will show that $4\gamma_T \neq 0$ whereas $4 \gamma_Q = 0$. Proposition \ref{prop: eight gamma nonzero} will combine these two estimates to deduce the nonvanishing of $8 \gamma$ in $K^{\sigma}_{C_2} (P)$.

\begin{lemma}
\label{lem: gamma_R}
The restriction of $\gamma$ to $R$ satisfies
    \begin{equation}
        \label{eq: two gamma R}
        2\gamma_R=\alpha\gamma_R.
    \end{equation}
\end{lemma}

\begin{proof}
For $1\leq j\leq k+l$, let
\begin{equation*}
    \Sigma_{j}=S^2\hookrightarrow B_2
\end{equation*}
be an embedding that represents the class $a_{j}$ and containing $B_0$. Let 
\begin{equation*}
    R'=\pi^{-1}_{1}(s(B_0))\cong S(3\sigma), \qquad R_{j}=\pi^{-1}_{1}(s(\Sigma_{j}))
\end{equation*}
and $\gamma_{R_j}=\gamma|_{R_{j}}$. We first prove that 
\begin{equation}\label{eq: 2gamma=alpha cdot gamma}
2\gamma_{R_j}=s_{j}\alpha_{j}\cdot \gamma_{R_{j}}\quad\text{for }1\leq j\leq k+l.    
\end{equation}
We focus on the case $s_{j}=1$ and $j=1$; the other cases are analogous. We have a bundle 
\begin{equation*}
    S(3\sigma)\hookrightarrow R_1{\to \Sigma_{1}},
\end{equation*}
and hence an Atiyah--Hirzebruch spectral sequence 
\begin{equation}
\label{eq: AHSS for Rk}
    E^{p,q}_{2}(\pi_0|_{R_1})=H^{p}(\Sigma_{1};K_{C_2}^{q+\sigma}(S(3\sigma)))\implies K^{p+q+\sigma}(R_{1})
\end{equation}
By Lemma \ref{lem: AHSS collapses} and naturality, the spectral sequence \eqref{eq: AHSS for Rk} also collapses on the second page. This gives a short exact sequence 
\begin{equation}
\label{eq: SES for K-R1}
\begin{tikzcd}
    K^{\sigma}_{C_2}(R_1,R')\arrow[r, hook]&  K^{\sigma}_{C_2}(R_1)\arrow[r,two heads]& K^{\sigma}_{C_2}(R') 
\end{tikzcd}  
\end{equation}
Moreover, 
\begin{equation*}
    K^{\sigma}_{C_2}(R_1,R')\cong \widetilde{K}^{\sigma}(S^2\wedge R'_{+})\cong \mathbb{Z}_2
\end{equation*}
is generated by $\alpha_{1}\cdot \gamma_{R_1}$ while
\begin{equation*}
    K^{\sigma}_{C_2}(R')\cong \mathbb{Z}_2
\end{equation*}
is generated $\gamma_{R'}$, the image of $\gamma_{R_1}$. Thus, to prove $2\gamma_{R_1}= \alpha_{1}\cdot \gamma_{R_1}$, it is enough to show that \eqref{eq: SES for K-R1} does not split, or equivalently that 
\begin{equation*}
    \mathbb{Z}_4\cong K^{\sigma}_{C_2}(R_1).
\end{equation*}

Since $R_{1}=P(W^{\perp}_{0}|_{\Sigma_1})$ and $s_1=1$, the $\mathbb{H}_{\mathcal{L}|_{\Sigma_1}}$-module  $W^{\perp}_{0}|_{\Sigma_1}$ is obtained by the clutching construction 
\begin{equation*}
W^{\perp}|_{\Sigma_1}=(D^2\times \mathbb{H})\sqcup (D^2\times \mathbb{H})/(e^{i\theta},h)\sim (e^{i\theta},e^{\frac{i\theta}{2}}he^{-\frac{i\theta}{2}}).
\end{equation*}
Since 
\begin{equation*}
e^{\frac{i\theta}{2}}(x+y\mathbf{j})e^{-\frac{i\theta}{2}}=x+e^{i\theta}y\mathbf{j}\quad \forall x,y\in \mathbb{C},
\end{equation*}
we obtain the isomorphism of complex vector bundles 
\begin{equation*}
    W^{\perp}|_{\Sigma_1}\cong \mathcal{O}(1)\oplus \underline{\mathbb{C}}_{\mathbb{CP}^1}.
\end{equation*}
Therefore, 
\begin{equation*}
    R_{1}=\mathbb{P}_{\mathbb{C}}(\mathcal{O}(1)\oplus \underline{\mathbb{C}}_{\mathbb{CP}^1})
\end{equation*}
is the fiberwise one-point compactification of $\mathcal{O}(1)\to \mathbb{CP}^1$. 

Let $S(\mathcal{O}(1))$ denote its unit circle bundle. As in the usual decomposition into the two disk bundles, $R_{1}$ is obtained from $S(\mathcal{O}(1))$ by attaching the upper- and lower-hemisphere bundles. Consequently, $(R_{1, +} \wedge S^\sigma)/C_2$ is obtained from $S(\mathcal{O}(1))_+ \wedge S^\sigma)/C_2$ by attaching a $3$-cell and a $5$-cell.

Since $S(\mathcal{O}(1)) \cong S^3$, the quotient  $(S(\mathcal{O}(1))_{+}\wedge  S^{\sigma})/C_2$ is the Thom space of the real line bundle associated to the double cover $S^3 \to \mathbb{RP}^3$. Hence it is pointed homotopy equivalent to $\mathbb{RP}^4$, and therefore
\begin{equation*}
    \widetilde{K}((S(\mathcal{O}(1))_{+}\wedge  S^{\sigma})/C_2) \cong \widetilde{K}(\mathbb{RP}^4) \cong \mathbb{Z}_4.
\end{equation*}
Since the remaining cells have odd dimensions, the long exact sequence in $K$-theory gives an injection
\begin{equation*}
    K^{\sigma}_{C_2}(R_1) \hookrightarrow \mathbb{Z}_4.
\end{equation*}
On the other hand, \eqref{eq: SES for K-R1} shows that $K^{\sigma}_{C_2}(R_1)$ has order four. This proves \eqref{eq: 2gamma=alpha cdot gamma}.

Since $\alpha_{j}|_{\Sigma_i}=0$ if $i\neq j$, the relation \eqref{eq: 2gamma=alpha cdot gamma} implies that $2\gamma_{R}-\alpha \gamma_{R}$ belongs to kernel of the restriction map 
\begin{equation*}
    f\colon K^{\sigma}_{C_2}(R)\longrightarrow \bigoplus_{1\leq j\leq k+l}  K^{\sigma}_{C_2}(R_j).
\end{equation*}
The map
\begin{equation*}
H^\bullet(R)\longrightarrow \bigoplus_{1\leq j\leq k+l} H^\bullet(R_j)
\end{equation*}
is injective. The Atiyah--Hirzebruch spectral sequences for $R$ and the $R_j$'s collapse on the second page, and the induced map on the associated graded groups is therefore injective. Hence $f$ is injective. This concludes the proof.
\end{proof}

\begin{lemma}\label{lem: gammaT order 8}
We have 
\begin{equation*}
    4\gamma_{T}\neq 0.
\end{equation*}
\end{lemma}

\begin{proof}
We have the bundle structure $S(3\sigma)\hookrightarrow T\to B$. This gives the spectral sequence 
\begin{equation*}
    E^{p,q}_{2}(\pi_0|_T)=H^{p}(B;K_{C_2}^{q+\sigma}(S(3\sigma)))\implies K_{C_2}^{p+q+\sigma}(T).   
\end{equation*}
By the same argument as in Lemma \ref{lem: AHSS collapses}, \footnote{Equivalently, this can be seen by the naturality of the spectral sequence applied to the section $s: B \to P_{\mathbb{H}}$.} this spectral sequence collapses on the second page. Hence we obtain the filtration 
\begin{equation*}
0=\mathcal{F}_{4}\subset \mathcal{F}_{2}\subset \mathcal{F}_{0}\subset \mathcal{F}_{-2}=K^{\sigma}_{C_2}(T),
\end{equation*}
where 
\begin{equation*}
    \mathcal{F}_{2}=\ker(K^{\sigma}_{C_2}(T)\to K^{\sigma}_{C_2}(R))\und{1.0cm}
    \mathcal{F}_{0}=\ker(K^{\sigma}_{C_2}(T)\to K^{\sigma}_{C_2}(R')).
\end{equation*}
We have $\mathcal{F}_{-2}/\mathcal{F}_0\cong \mathbb{Z}_2$, generated by the image of $\gamma_{T}$ and $\mathcal{F}_{2}=\mathbb{Z}_2$, generated by $\alpha^2\gamma$. Indeed, in the associated graded group, the class $\alpha^2 \gamma_T$ corresponds to $a^2 \cdot \gamma_{R'} \in H^4(B; \mathbb{Z}_2)$, which is nonzero since $a^2 \neq 0$.

By Lemma \ref{lem: gamma_R}, we have 
\begin{equation*}
    (2-\alpha)\, \gamma_{T}|_{R}=0
\end{equation*}
and hence $(2-\alpha)\,\gamma \in \mathcal{F}_2$. Therefore, for some $s\in \{0, 1\}$, 
\begin{equation*}
    2\gamma_{T}=(\alpha+s\alpha^2)\, \gamma_{T}.
\end{equation*}
It follows that 
\begin{equation*}
    4\gamma_{T}=(\alpha+s\alpha^2)^2 \, \gamma_{T}=\alpha^2 \gamma_{T}\neq 0,
\end{equation*}
where the second equality follows from $\alpha^3 = 0$. This proves the lemma.
\end{proof}

We now turn to the restriction of $\gamma$ to $Q$. We first analyze its restriction to the distinguished fiber $Q'$ which will be used to control the filtration of $K^\sigma_{C_2}(Q)$.

Note that $Q'\cong \mathbb{CP}^3$, with the $C_2$-action given by\footnote{This follows from the equality $j\cdot (x+yj)=-\bar{y}+\bar{x}j$ for $x,y\in \mathbb{C}$.}
\begin{equation*}
[x,y,x',y']\mapsto [-\bar{y},\bar{x}, -\bar{y}',\bar{x}'].
\end{equation*} 

\begin{lemma}
\label{lem: extension 2}
The isomorphism $K^{\sigma}_{C_2}(Q')\cong \mathbb{Z}_4$ is generated by $\gamma_{Q'}$.
\end{lemma}

\begin{proof}
We set $A=Q'_{+}$ in \eqref{eq: exact 6-gon v1} and get the exact sequence 
\begin{equation}
\label{eq: exact sequence Q}
\begin{tikzcd}
	& {\widetilde{K}^\sigma_{C_2}(Q')} & {\widetilde{K}_{C_2}(Q')} & \\
	{\widetilde{K}^1_{C_2}(Q')} &&& {\widetilde{K}(Q')} \\
	& {\widetilde{K}^1_{C_2}(Q')} & {\widetilde{K}^{1+\sigma}_{C_2}(Q')}
	\arrow[from=1-2, to=1-3]
	\arrow[from=1-3, to=2-4]
	\arrow["{\mathrm{tr}}"{description}, from=2-1, to=1-2]
	\arrow[from=2-4, to=3-3]
	\arrow[from=3-2, to=2-1]
	\arrow[from=3-3, to=3-2]
\end{tikzcd}
\end{equation}

Note that 
\begin{equation*}
\begin{tikzcd}
    K_{C_2}(Q')\cong K_{\m{Pin(2)}}(S(2\mathbb{H}))=\m{\coker}(R(\m{Pin(2)}))\arrow[rr,"\cdot h^2"] &&R(\m{Pin(2)}))\cong \mathbb{Z}\oplus \mathbb{Z},
\end{tikzcd}
\end{equation*}
where $R(\Pin(2))$ is the complex representation ring of $\Pin(2)$. Set $w=[\widetilde{\mathbb{C}}]-1$ and $h=[\mathbb{H}]-2$. Then we have  
\begin{equation*}
    R(\m{Pin(2)})\cong \mathbb{Z}[h,w]/(w^2=2w, hw=2w).
\end{equation*} This implies $K_{C_2}(Q')\cong \mathbb{Z}_4\oplus \mathbb{Z}\oplus \mathbb{Z}$. Since $Q'$ only contains cells of even dimensions, we can use the Atiyah--Hirzebruch spectral sequence to conclude that $K^{1}(Q')=0$ and $K(Q')\cong \mathbb{Z}^4$. Since $K(Q')$ is torsion-free, the $\mathbb{Z}_4$-torsion summand of $K_{C_2}(Q')$ lies in the kernel of the map $K_{C_2}(Q') \to K(Q')$. Hence, by exactness of \eqref{eq: exact sequence Q} $K^{\sigma}_{C_2}(Q')$ contains an element of order at least four. 

On the other hand, we have the bundle $S(3\sigma)\hookrightarrow Q'\to S^4$ pulled back from the bundle $\pi_{1}: P(W^{\perp})\to P_{\mathbb{H}}(W^{\perp})$. This gives the Atiyah--Hirzebruch spectral sequence 
\begin{equation*}
    E^{p,q}_{2}(\pi_1|_{Q'})=H^{p}(S^4;K^{q+\sigma}(S(3\sigma)))\implies K^{p+q+\sigma}_{C_2}(Q').
\end{equation*}
By the same argument as in Lemma \ref{lem: AHSS collapses}, it collapses on the second page. This spectral sequence provides a short exact sequence 
\begin{equation*}
    \begin{tikzcd}
        \mathbb{Z}_2\arrow[r,hook]& K^{\sigma}_{C_2}(Q')\arrow[r,two heads, "i^*"]& \mathbb{Z}_2.
    \end{tikzcd}
\end{equation*}
Together with the existence of an element of order at least four, this implies 
\begin{equation*}
    K^{\sigma}_{C_2}(Q')\cong \mathbb{Z}_4.
\end{equation*}
Since $\gamma_{Q'}$ is nonzero under the fiber inclusion map
\begin{equation*}
    i^*\colon K^{\sigma}_{C_2}(Q')\to K^{\sigma}_{C_2}(S(3\sigma)) \cong \mathbb{Z}_2,
\end{equation*}
it has order four and hence is a generator of $K^{\sigma}_{C_2}(Q')$.
\end{proof}

We now pass from the distinguished fiber $Q'$ to the whole space $Q$. To describe the resulting Atiyah--Hirzebruch filtration, we first identify the $K$-theory of the base $\widetilde{B}_2$ and construct a class detecting its $S^4$-direction.

Note that $B_2\simeq\vee_{k+l} S^2$, so we have 
\begin{equation*}
K(B_2)=\mathbb{Z}\langle\alpha_1,\cdots,\alpha_{k+l}\rangle/(\alpha_{i}\alpha_{j})_{1\leq i,j\leq k+l}.
\end{equation*}
We let $\widetilde{B}_{2}=\pi^{-1}_{2}(B_2)\subset P_{\mathbb{H}}$. Since $B_2$ deformation retracts to $\vee_{k+l}S^2$, the  space $\widetilde{B}_2$ deformation retracts to a 6-dimensional CW complex $\widetilde{B}'_2$, which is $S^4$-bundle over $\vee_{k+l}S^2$. The obvious CW structure on $B'_{2}$ induces a CW structure on $\widetilde{B}'_2$, which a single $0$-cell, $(k+l)$-many $2$-cells, a single $4$-cell and $(k+l)$-many $6$-cells. We use $\m{sk}_{p}\widetilde{B}'_2$ to denote the $p$-skeleton.

We have a bundle $S(3\sigma)\hookrightarrow Q\to \widetilde{B}_{2}$, making $K^{\sigma}_{C_2}(Q)$ into a module over $K(\widetilde{B}_{2})$.

\begin{lemma}
\label{lem: K(B) ring}
 There exists a class $\beta\in K(\widetilde{B}_{2})$ that satisfies the following properties:
 \begin{enumerate}
     \item $\beta|_{S^4}$ is a generator of $\widetilde{K}(S^4)$;
     \item $\beta|_{s(B_2)}=0$;
     \item The $K(B_2)$-algebras  $K(\widetilde{B}_{2})$ and $K(B_2)[\beta]/(\beta^2)$ are isomorphic.
 \end{enumerate}
\end{lemma}
\begin{proof} Note that 
$\widetilde{B}_{2}/s(B_2)\cong \m{Th}(W^{\perp}_{1}|_{B_2}),
$
where $W^{\perp}_{1}$ is the orthogonal complement of $W^{\perp}_0$ in $W^{\perp}$. Since $W^{\perp}_{1}\to B$ is a complex vector bundle, there exists a Thom class $\beta_0\in K(\widetilde{B}_{2},s(B_2))$ which restricts to a generator of $\widetilde{K}(S^4)$. 

Let $\beta\in K(\widetilde{B}_2)$ be the image of $\beta_0$. By construction $\beta|_{s(B_2)} = 0$ while $\beta|_{S^4}$ is a generator of $\widetilde{K}(S^4)$.

Moreover, the map 
\begin{equation*}
    s^*: H^2(\widetilde{B}_2; \mathbb{Z}) \longrightarrow H^2(B_2; \mathbb{Z})
\end{equation*}
is an isomorphism since $\pi^*_2: H^2(B_2; \mathbb{Z}) \to H^2(\widetilde{B}_2; \mathbb{Z})$. Since $s^* c_1(\beta) = c_1(s^*\beta) =0$, it follows that $c_1(\beta) = 0$.

By the Leray-Hirsch theorem, we have the isomorphism
\begin{equation}
\label{eq: K(PH) as module}
    K(\widetilde{B}_{2})\cong K(B_2)\langle 1,\beta\rangle    
\end{equation}
between $K(B_2)$-modules. It remains to show that $\beta^2=0$. 

Since $\beta|_{s(B_2)} = 0$, the class $\beta$ has virtual rank zero, and since $c_1(\beta) = 0$, we have 
\begin{equation*}
    \m{ch}(\beta)=-c_{2}(\beta)+\frac{c_{3}(\beta)}{2}.
\end{equation*}
Hence $\m{ch}(\beta^2)=0$ for dimensional reasons. Thus $\beta^2$ is torsion in $K(\widetilde{B}_{2})$. However, by \eqref{eq: K(PH) as module}, the group $K(\widetilde{B}_{2})$ is torsion free. Therefore $\beta^2=0$.
\end{proof}

\begin{lemma}\label{lem: gammaQ order 4}
It holds that 
\begin{equation*}
    4\gamma_{Q}= 0.
\end{equation*}
\end{lemma}
\begin{proof} Consider the bundle $Q\to \widetilde{B}_{2}$. The CW filtration on $\widetilde{B}_{2}$ induces a filtration
\begin{equation*}
    0=\mathcal{F}_{6}\subset \mathcal{F}_{4}\subset \mathcal{F}_{2}\subset\mathcal{F}_{0}\subset \mathcal{F}_{-2}=K^{\sigma}_{C_2}(Q)
\end{equation*}
where  
\begin{equation*}
    \begin{split}
    \mathcal{F}_{0}&=\ker(K^{\sigma}_{C_2}(Q)\to K^{\sigma}_{C_2}(R')),\\ \mathcal{F}_{2}&=\ker(K^{\sigma}_{C_2}(Q)\to K^{\sigma}_{C_2}(R)),\\
    \mathcal{F}_{4}&=\ker(K^{\sigma}_{C_2}(Q)\to K^{\sigma}_{C_2}(Q'\cup R)).
\end{split}
\end{equation*}
Consider the associated Atiyah--Hirzebruch spectral sequence
\begin{equation*}
    E^{p,q}_{2}(\pi_1|_Q)=H^{p}(\widetilde{B}_2;K^{q+\sigma}_{C_2}(S(3\sigma)))\implies K^{p+q+\sigma}_{C_2}(Q).
\end{equation*}
By the same argument as in Lemma \ref{lem: AHSS collapses}, this spectral sequence collapses on the second page. Hence 
\begin{equation*}
\begin{split}
    \mathcal{F}_{-2}/\mathcal{F}_0&=\mathbb{Z}_2\langle \gamma_{Q}\rangle,\\
    \mathcal{F}_{0}/\mathcal{F}_{2}&=\mathbb{Z}_2\langle \alpha_{1}\gamma_{Q}, \cdots,\alpha_{k+l}\gamma_{Q} \rangle,\\
    \mathcal{F}_{2}/\mathcal{F}_{4}&=\mathbb{Z}_2\langle \beta\gamma_{Q}\rangle,\\
    \mathcal{F}_{4}&=\mathbb{Z}_2\langle \alpha_{1}\beta\gamma_{Q},\cdots,\alpha_{k+l}\beta\gamma_{Q}.\rangle
\end{split}
\end{equation*}
By Lemma \ref{lem: gamma_R}, 
\begin{equation*}
    (2-\alpha) \, \gamma_{Q}\in \mathcal{F}_2
\end{equation*}
and by Lemma \ref{lem: extension 2},  
\begin{equation*}
    \big((2-\alpha)\,  \gamma_{Q}\big)\Big|_{Q'}=2\gamma_{Q'}\neq 0.
\end{equation*}
Hence
\begin{equation*}
    (2-\alpha)\,  \gamma_{Q}\notin \mathcal{F}_{4}.
\end{equation*}
Since $\mathcal{F}_2/\mathcal{F}_4$ is generated by $\beta \gamma_Q$, it follows that 
\begin{equation*}
    (2-\alpha-\beta) \, \gamma_{Q}\in \mathcal{F}_{4}.
\end{equation*}
Thus 
\begin{equation*}
    2\gamma_{Q}=(\alpha+\beta+\beta \alpha')\, \gamma_{Q},
\end{equation*}
where $\alpha'$ is linear combination of $\alpha_1,\cdots,\alpha_{k+l}$. Consequently 
\begin{equation*}
    4\gamma_{Q}=(\alpha+\beta+\beta\alpha')^2 \, \gamma_{Q}=2 \alpha\beta\gamma_{Q}=0.
\end{equation*}
The second equality uses $\beta^2=0$ from Lemma \ref{lem: K(B) ring} together with $\alpha_i \alpha_j = 0$ in $K(B_2)$ for $1 \leq i, j \leq k+l$, and the last equality follows because $\mathcal{F}_4$ is annihilated by $2$.  
\end{proof}

We have thus obtained the two complementary estimates
\begin{equation*}
    4 \gamma_T \neq 0 \quad \text{and} \quad 4\gamma_Q = 0.
\end{equation*}
We now combine them using the decompositions of $P$ determined by $T$ and $Q$.
\begin{prop}\label{prop: eight gamma nonzero}
For any $\gamma\in K^{\sigma}_{C_2}(P)$ such that $i^*_{1}(\gamma)\neq 0$, we have $8\gamma\neq 0$.   
\end{prop}
\begin{proof}
We consider the commutative diagram 
\begin{equation*}
    \begin{tikzcd}
        R\arrow[r,hook]\arrow[d,hook]&T\arrow[r,two heads]\arrow[d]&T/R\arrow[d,hook]\\
        Q\arrow[r,hook]&P\arrow[r,two heads]&P/Q
    \end{tikzcd}
\end{equation*}
which induces the commutative diagram
\begin{equation}
\label{eq: diagram of K(subspace P)}
    \begin{tikzcd}
	{K^\sigma_{C_2}(R)} && {K^\sigma_{C_2}(T)} && {K^\sigma_{C_2}(T/R)\cong\mathbb{Z}_2} \\
	\\
	{K^\sigma_{C_2}(Q)} && {K^\sigma_{C_2}(P)} && {K^\sigma_{C_2}(P/Q)\cong\mathbb{Z}_4}
	\arrow[two heads, from=1-3, to=1-1]
	\arrow[hook', from=1-5, to=1-3]
	\arrow[two heads, from=3-1, to=1-1]
	\arrow["{f_2}"{description}, from=3-3, to=1-3]
	\arrow["{f_1}"{description}, two heads, from=3-3, to=3-1]
	\arrow["{f_4}"{description}, two heads, from=3-5, to=1-5]
	\arrow["{f_3}"{description}, hook', from=3-5, to=3-3]
\end{tikzcd}
\end{equation}


The two rows are exact. We justify the additional claims in \eqref{eq: diagram of K(subspace P)}:
\begin{itemize}
    \item $K^{\sigma}_{C_2}(T/R)\cong \mathbb{Z}_2$ because  $T/R\cong S^4\wedge S(3\sigma)_{+}$;
    \item $K^{\sigma}_{C_2}(P/Q)\cong \mathbb{Z}_4$ because $P/Q\cong S^4\wedge Q'_{+}$;
    \item The map $f_{3}$ is injective. Indeed, by the long exact sequence of the pair $(P, Q)$, it suffices to show that $K^{\sigma-1}_{C_2}(P) \to K^{\sigma-1}_{C_2}(Q)$ is surjective. The induced map of the Atiyah--Hirzebruch spectral sequence \eqref{eq: AHSS} is surjective on the second page, and both spectral sequences collapse by Lemma \ref{lem: AHSS collapses}. Hence the restriction map in $K$-theory is surjective.
\end{itemize}

By Lemma \ref{lem: gammaQ order 4}, we have $f_{1}(4\gamma)=4\gamma_{Q}=0$. By exactness of the second row, there exists $\delta\in K^{\sigma}_{C_2}(P/Q)$ such that $f_{3}(\delta)=4\gamma$. By Lemma \ref{lem: gammaT order 8}, we have $f_{2}(4\gamma)=4\gamma_{T}\neq 0$. By commutativity of the diagram, this implies $f_{4}(\delta)\neq 0$. So $\delta$ must be a generator of $K^{\sigma}_{C_2}(P/Q)$. Since $2\delta\neq 0$ and $f_{3}$ is injective, we have $8\gamma=f_{3}(2\delta)\neq 0$.
\end{proof}

\subsection{Completion of the proof}
Now we return to the map 
\begin{equation*}
    f\colon S^{\sigma}\wedge P_{+}\xrightarrow{\m{Id}\wedge q_{P}} S^{\sigma} \wedge S^{0}=S^{\sigma}\hookrightarrow  S^{8\sigma}  
\end{equation*}
as defined in \eqref{eq: P to S}. 
\begin{prop}\label{prop: f induces nontrivial map on K theory} The induced map
\begin{equation*}
    f^*\colon\widetilde{K}_{C_2}(S^{\sigma})\to \widetilde{K}_{C_2}(S^{\sigma}\wedge P_{+})
\end{equation*}
is nontrivial.
\end{prop}
\begin{proof}
We have the following commutative diagram 
\begin{equation*}
    \begin{tikzcd}
	&&&& {S^{8\sigma}} \\
	\\
	{S^\sigma\wedge S(3\sigma)} && {S^\sigma \wedge P_+} && {S^\sigma} \\
	\\
	{S(3\sigma)} && {P_+} && {S^0}
	\arrow["{\mathrm{Id}\wedge i_1}"{description}, hook, from=3-1, to=3-3]
	\arrow["f"{description}, from=3-3, to=1-5]
	\arrow["{\mathrm{Id}\wedge q_P}"{description}, two heads, from=3-3, to=3-5]
	\arrow["{j_2}"{description}, hook, from=3-5, to=1-5]
	\arrow["{j_1\wedge\mathrm{Id}}"{description}, hook, from=5-1, to=3-1]
	\arrow["{i_1}"{description}, hook, from=5-1, to=5-3]
	\arrow["{j_1\wedge\mathrm{Id}}"{description}, hook, from=5-3, to=3-3]
	\arrow["{q_P}"{description}, two heads, from=5-3, to=5-5]
	\arrow["{j_1}"{description}, hook, from=5-5, to=3-5]
\end{tikzcd}
\end{equation*}

Here $j_{1},j_{2}$ are standard embeddings of representation spheres. Applying the functor $\widetilde{K}_{C_2}$, we obtain the commutative diagram
\begin{equation}
    \label{eq: diagram K}
    \begin{tikzcd}
	&&&& {\widetilde{K}_{C_2}(S^{8\sigma})\cong R(C_2)} \\
	\\
	{\mathbb{Z}_2\cong K^\sigma_{C_2}(S(3\sigma))} && {K^\sigma_{C_2}(P)} && {\widetilde{K}^\sigma_{C_2}(S^0)\cong\mathbb{Z}\cdot\theta} \\
	\\
	{\mathbb{Z}\oplus\mathbb{Z}_2\cong K_{C_2}(S(3\sigma))} && {K_{C_2}(P)} && {\widetilde{K}_{C_2}(S^0)\cong R(C_2)} 
	\arrow["{f^*}"{description}, from=1-5, to=3-3]
	\arrow["{j_2^*}"{description},  from=1-5, to=3-5]
	\arrow["{j_1^*}"{description},  from=3-1, to=5-1]
	\arrow["{i_1^*}"{description},  from=3-3, to=3-1]
	\arrow["{j_1^*}"{description},  from=3-3, to=5-3]
	\arrow[hook', "{q_P^*}"{description}, from=3-5, to=3-3]
	\arrow["{j_1^*}"{description},  from=3-5, to=5-5]
	\arrow["{i_1^*}"{description},  from=5-3, to=5-1]
	\arrow[hook',"{q_P^*}"{description}, from=5-5, to=5-3]
    \end{tikzcd}
\end{equation}

We explain the additional claims in \eqref{eq: diagram K}.
\begin{itemize}
    \item By the Bott periodicity theorem, we have $\widetilde{K}_{C_2}(S^{8\sigma})\cong R(C_2)$. Furthermore, the map \begin{equation*}j^*_{1}\circ j^*_{2}\colon R(C_2)\to R(C_2)\end{equation*} is just the multiplication by the $K$-theoretic Euler class  
    \begin{equation*}
    e_{8\sigma}=(1-\sigma_{\mathbb{C}})^{4}=8(1-\sigma_{\mathbb{C}})
    \end{equation*}
    Here $\sigma_{\mathbb{C}}=\sigma\otimes \m{Id}_{\mathbb{C}}$ is the complex sign representation. In particular, we have 
    \begin{equation}
        \label{eq: image of 1 equals euler class}
        j^*_{1}\circ j^*_{2}(1)=8(1-\sigma_{\mathbb{C}})    
    \end{equation}
    \item We set  $A=S^0$ in the exact sequence \eqref{eq: exact 6-gon v1} and get 
    \begin{equation}
    \label{eq: exact 6-gon v2}
        \begin{tikzcd}
	& {\widetilde{K}_{C_2}(S^0)\cong R(C_2)} && {\widetilde{K}(S^0)\cong\mathbb{Z}} & \\
	{\widetilde{K}^\sigma_{C_2}(S^0)} &&&& {\widetilde{K}^{1+\sigma}_{C_2}(S^0)} \\
	& {0=\widetilde{K}^1(S^0)} && {\widetilde{K}^1_{C_2}(S^0)}
	\arrow["{\mathrm{rank}}"{description}, from=1-2, to=1-4]
	\arrow[from=1-4, to=2-5]
	\arrow["{j_1^*}"{description}, from=2-1, to=1-2]
	\arrow[from=2-5, to=3-4]
	\arrow[from=3-2, to=2-1]
	\arrow[from=3-4, to=3-2]
\end{tikzcd}
    \end{equation}
    
Therefore, the group $\widetilde{K}^{\sigma}_{C_2}(S^0)$ is a free abelian group generated by an element $\theta$ with 
\begin{equation}
\label{eq: j1 theta}
    j^*_{1}(\theta)=1-\sigma_{\mathbb{C}}.
\end{equation}
By \eqref{eq: image of 1 equals euler class} and \eqref{eq: j1 theta}, we have 
\begin{equation}
\label{eq: image of j2}
    j^*_{2}(1)=8\theta.  
\end{equation}
\item Since the $C_2$-action on $S(3\sigma)$ is free, we have
\begin{equation*}
    K_{C_2}(S(3\sigma))\cong K(S(3\sigma)/C_2)=K(\mathbb{RP}^2)\cong \mathbb{Z}\oplus \mathbb{Z}_2.
\end{equation*}
\item The map 
\begin{equation*}
    i_1^*\circ q^*_{2}\colon R(C_2)\to K_{C_2}(S(3\sigma))\cong K(\mathbb{RP}^2)
\end{equation*}
is induced by the constant map $S(3\sigma)\to \ast$. Hence by definition, the element $i_1^*\circ q^*_{2}(\sigma_{\mathbb{C}})\in K(\mathbb{RP}^2)$ is represented by the bundle 
\begin{equation*}
    \mathbb{C}\hookrightarrow S(3\sigma)\times_{C_2} \sigma_{\mathbb{C}} \to  \mathbb{RP}^2
\end{equation*}
The first Chern class of this bundle is the generator of $H^{2}(\mathbb{RP}^2)=\mathbb{Z}_2$. Hence the bundle is not stably trivial. This implies 
\begin{equation}
\label{eq: i q_P sigma nonzero}
    i_1^*\circ q^*_{2}(\sigma_{\mathbb{C}}-1)\neq 0.    
\end{equation}
\end{itemize}
We let $\gamma=q^*_{2}(\theta)$. Then by \eqref{eq: j1 theta}, \eqref{eq: i q_P sigma nonzero} and commutativity of the  diagram \eqref{eq: diagram K}, we see that  $i_1^*(\gamma)\neq 0$. Therefore, we can apply Proposition \ref{prop: eight gamma nonzero} and deduce that $8\gamma\neq 0$. On the other hand, by \eqref{eq: image of j2}, we have $f^{*}(1)=8\gamma$. Hence the map $f^*$ is nontrivial. This finishes the proof.
\end{proof}
We can finally finish the proof of Theorem \ref{mainthm: X non spin}.
\begin{proof}[Proof of Theorem \ref{mainthm: X non spin} in the case $n=2$] Suppose there exists a smooth bundle  $F\hookrightarrow X\to B$ such that $w_{2}(V\pi)^2\neq 0$. Then the desuspended Seiberg--Witten map $\mathsf{sw}_B$ exists. By Lemma \ref{lem: sw map implies null homotopic}, the map $f$ is null homotopic. But this contradicts Proposition \ref{prop: f induces nontrivial map on K theory}.
\end{proof}

\bibliographystyle{plain}  
\bibliography{references}
\end{document}